\date{\today} 
\documentclass[10pt]{article}
\usepackage{hyperref}
\usepackage{amssymb, amsmath}
\usepackage{stmaryrd}
\usepackage[final]{showlabels}
\usepackage{float}
\usepackage{abstract}

\usepackage{latexsym}

\usepackage{amsthm}

\newcommand{\g}{{\mathfrak g}}

\newcommand{\fa}{{\mathfrak a}}

\newcommand{\fe}{{\mathfrak e}}
\newcommand{\ff}{{\mathfrak f}}
\newcommand{\fg}{{\mathfrak g}}
\newcommand{\fh}{{\mathfrak h}}

\newcommand{\fk}{{\mathfrak k}}

\newcommand{\fn}{{\mathfrak n}}

\newcommand{\fp}{{\mathfrak p}}

\newcommand{\fs}{{\mathfrak s}}
\newcommand{\ft}{{\mathfrak t}}
\newcommand{\fu}{{\mathfrak u}}

\newcommand{\fz}{{\mathfrak z}}

\renewcommand\sp{\mathfrak {sp}}

\newcommand{\1}{\mathbf{1}}

\newcommand{\cA}{\mathcal{A}}

\newcommand{\cH}{\mathcal{H}}

\newcommand{\cM}{\mathcal{M}}

\newcommand{\cO}{\mathcal{O}}

\newcommand{\cW}{\mathcal{W}}

\newcommand{\N}{{\mathbb N}}
\newcommand{\Z}{{\mathbb Z}}
\newcommand{\R}{{\mathbb R}}
\newcommand{\C}{{\mathbb C}}

\newcommand{\K}{{\mathbb K}}

\renewcommand{\H}{{\mathbb H}}

\renewcommand{\hat}{\widehat}

\renewcommand{\tilde}{\widetilde}

\renewcommand{\L}{\mathop{\bf L{}}\nolimits}

\newcommand{\GL}{\mathop{{\rm GL}}\nolimits}

\newcommand{\OO}{\mathop{\rm O{}}\nolimits}

\newcommand{\Sym}{\mathop{{\rm Sym}}\nolimits}

\newcommand{\gl}  {\mathop{{\mathfrak{gl} }}\nolimits}

\newcommand{\fsl} {\mathop{{\mathfrak{sl} }}\nolimits}

\newcommand{\su}  {\mathop{{\mathfrak{su} }}\nolimits}
\newcommand{\so}  {\mathop{{\mathfrak{so} }}\nolimits}

\newcommand{\Fix}{\mathop{{\rm Fix}}\nolimits}

\newcommand{\ad}{\mathop{{\rm ad}}\nolimits}
\newcommand{\Ad}{\mathop{{\rm Ad}}\nolimits}

\newcommand{\tr}{\mathop{{\rm tr}}\nolimits}
\newcommand{\str}{\mathop{{\rm str}}\nolimits}

\newcommand{\Herm}{\mathop{{\rm Herm}}\nolimits}

\newcommand{\Aut}{\mathop{{\rm Aut}}\nolimits}

\newcommand{\Der}{\mathop{{\rm Der}}\nolimits}

\newcommand{\diag}{\mathop{{\rm diag}}\nolimits}
\newcommand{\End}{\mathop{{\rm End}}\nolimits}
\newcommand{\id}{\mathop{{\rm id}}\nolimits}
\newcommand{\rk}{\mathop{{\rm rank}}\nolimits}

\renewcommand{\dim}{\mathop{{\rm dim}}\nolimits}

\newcommand{\spec}{\mathop{{\rm spec}}\nolimits}

\newcommand{\Inn}{\mathop{{\rm Inn}}\nolimits}

\newcommand{\cone}{\mathop{{\rm cone}}\nolimits}

\newcommand{\oline}{\overline}
\newcommand{\la}{\langle}
\newcommand{\ra}{\rangle}
\newcommand{\spann}{{\rm span}}

\def\theoremname{Theorem}
\def\propositionname{Proposition}
\def\corollaryname{Corollary}
\def\lemmaname{Lemma}
\def\remarkname{Remark}
\def\conjecturename{Conjecture} 

\def\definitionname{Definition}
\def\exercisename{Exercise}
\def\examplename{Example}
\def\examplesname{Examples}
\def\problemname{Problem}
\def\problemsname{Problems}

\def\@thmcounter#1{\noexpand\arabic{#1}}
\def\@thmcountersep{}
\def\@begintheorem#1#2{\it \trivlist \item[\hskip 
\labelsep{\bf #1\ #2.\quad}]}
\def\@opargbegintheorem#1#2#3{\it \trivlist
      \item[\hskip \labelsep{\bf #1\ #2.\quad{\rm #3}}]}
\makeatother
\newtheorem{theor}{\theoremname}[section]
\newtheorem{propo}[theor]{\propositionname}
\newtheorem{coro}[theor]{\corollaryname}
\newtheorem{lemm}[theor]{\lemmaname}

\newenvironment{thm}{\begin{theor}\it}{\end{theor}}

\newenvironment{prop}{\begin{propo}\it}{\end{propo}}

\newenvironment{cor}{\begin{coro}\it}{\end{coro}}

\newenvironment{lem}{\begin{lemm}\it}{\end{lemm}}
\newenvironment{lemma}{\begin{lemm}\it}{\end{lemm}}

\newtheorem{rema}[theor]{\remarkname}

\newenvironment{rem}{\begin{rema}\rm}{\end{rema}}

\newtheorem{stepnow}[theor]{}

\newtheorem{defin}[theor]{\definitionname} 

\newenvironment{definition}{\begin{defin}\rm}{\end{defin}}

\newtheorem{exerc}{\exercisename}[section]

\newtheorem{exa}[theor]{\examplename}

\newenvironment{example}{\begin{exa}\rm}{\end{exa}}
\newenvironment{ex}{\begin{exa}\rm}{\end{exa}}

\newtheorem{exas}[theor]{\examplesname}

\newtheorem{conj}[theor]{\conjecturename}

\newtheorem{pro}[theor]{\problemname}

\newtheorem{prs}[theor]{\problemsname}

\newcommand{\pmat}[1]{\begin{pmatrix} #1 \end{pmatrix}}

\usepackage{bbold}

\newcommand{\Wmin}{W_{\mathrm{min}}}
\newcommand{\Wmax}{W_{\mathrm{max}}}
\newcommand{\Cmin}{C_{\mathrm{min}}}
\newcommand{\Cmax}{C_{\mathrm{max}}}

\newcommand{\Stand}{\mathrm{Stand}}
\newcommand{\bbone}{\mathbb{1}}
\newcommand{\bbO}{\mathbb{O}}

\newcommand{\p}{\partial}

\newcommand{\tran}{\mathsf{T}}

\newcommand{\der}{\mathrm{der}}

\usepackage{enumitem}
\setlist[enumerate,1]{label={\rm (\alph*)}}
\setlist[enumerate,2]{label={\rm (\roman*)}}

\makeatletter
\def\blfootnote{\xdef\@thefnmark{}\@footnotetext}
\makeatother

\begin{document}

\title{Classification of 3-graded causal subalgebras of real simple Lie algebras}

\author{Daniel Oeh\footnote{Department Mathematik, Friedrich--Alexander--Universit\"at Erlangen--N\"urnberg, Cauerstr. 11, D-91058 Erlangen, Germany, oehd@math.fau.de} \,\footnote{Supported by DFG-grant Ne 413/10-1}} 

\maketitle

\begin{NoHyper}
\end{NoHyper}

\begin{abstract}
  Let \((\g,\tau)\) be a real simple symmetric Lie algebra and let \(W \subset \g\) be an invariant closed convex cone which is pointed and generating with \(\tau(W) = -W\). For elements \(h \in \g\) with \(\tau(h) = h\), we classify the Lie algebras \(\g(W, \tau,h)\) which are generated by the closed convex cones
  \[C_{\pm}(W, \tau, h) := (\pm W) \cap \g_{\pm 1}^{-\tau}(h),\]
  where \(\g^{-\tau}_{\pm 1}(h) := \{x \in \g : \tau(x) = -x, [h,x] = \pm x\}\). These cones occur naturally as the skew-symmetric parts of the Lie wedges of endomorphism semigroups of certain standard subspaces.

  We prove in particular that, if \(\g(W, \tau,h)\) is non-trivial, then it is either a hermitian simple Lie algebra of tube type or a direct sum of two Lie algebras of this type. Moreover, we give for each hermitian simple Lie algebra and each equivalence class of involutive automorphisms \(\tau\) of \(\g\) with \(\tau(W) = -W\) a list of possible subalgebras \(\g(W, \tau,h)\) up to isomorphy.
\end{abstract}

\section{Introduction}
\label{sec:intro}

Consider a quadruple \((\g, W, \tau, D)\), where \(\g\) is a real finite-dimensional Lie algebra, \(W \subset \g\) is a pointed generating \((e^{\ad \g})\)-invariant closed convex cone, \(\tau \in \Aut(\g)\) is an involution, and \(D \in \der(\g)\) is a derivation of \(\g\).
Suppose that \(\tau\) and \(D\) commute, that \(\tau(W) = -W\), and that \(e^{\R D}W \subset W\).
For an \(\ad(x)\)-invariant subspace \(F \subset \g, x \in \g,\) and \(\lambda \in \R\), we define \(F_\lambda(x) := \ker(\ad(x) - \lambda \id_F)\).
Moreover, we define
\[C_\pm(W, \tau, D) := (\pm W) \cap \g^{-\tau}_{\pm 1}(D) \quad \text{and} \quad \g_\pm(W, \tau, D) := \spann\,C_\pm(W, \tau, D).\]
One can show (cf.\ Lemma \ref{lem:gred-subalg}) that the subspaces \(\g_\pm(W, \tau, D)\) are abelian, so that
\begin{equation}
  \label{eq:std-wedge-ideal}
  \g(W, \tau, D) := \g_-(W, \tau, D) \oplus [\g_-(W, \tau, D), \g_+(W, \tau, D)] \oplus \g_+(W, \tau, D)
\end{equation}
is a subalgebra of \(\g\).
We call \(\g(W, \tau, D)\) the \emph{3-graded causal subalgebra defined by \((W, \tau, D)\)}, and we say that \((\g, W, \tau, D)\) is a \emph{3-graded causal Lie algebra} if
\(C_+(W,\tau,D) \cup C_-(W, \tau, D)\) generates \(\g\), that is, we have \(\g = \g(W, \tau, D)\).
If \(D = \ad h\) for some \(h \in \g\), then we also write
\[\g(W, \tau, h) := \g(W, \tau, \ad h), \quad \g_\pm(W, \tau, h) := \g_\pm(W, \tau, \ad h), \quad \text{and} \quad C_\pm(W, \tau, h) := C_\pm(W, \tau, \ad h) .\]

In this article, we classify for all real finite-dimensional simple Lie algebras \(\g\) the 3-graded causal subalgebras \(\g(W, \tau, D)\) up to isomorphy.

Before stating the main theorem of this article, we make some preliminary observations. Let \(\g\) be a simple Lie algebra and let \(W \subset \g\) be an invariant cone. As the cone \(W\) is invariant under inner automorphisms of \(\g\) by assumption, the subspaces \(W \cap (-W)\) and \(W - W\) are ideals in \(\g\).
If \(W\) is non-trivial, i.e.\ \(\{0\} \neq W \neq \g\), then \(\g\) is called a \emph{hermitian simple} Lie algebra.
Every hermitian simple Lie algebra \(\g\) contains an up to a sign unique pointed generating invariant closed convex cone \(\Wmin\) such that, for every pointed generating invariant closed convex cone \(W \subset \g\), we either have \(\Wmin \subset W\) or \(\Wmin \subset -W\).
This is a consequence of the Kostant--Vinberg Theorem (cf.\ \cite[p.\ 262]{HHL89}).
We will later see (cf.\ Lemma \ref{lem:assoc-wedge-simple}) that we have \(\g_\pm(W,\tau,h) = \g_\pm(\Wmin,\tau,h)\) if \(\tau(W) = -W\).
This means that \(\g(W, \tau, h)\) does not depend on the choice of \(W\) if \(W\) is pointed and generating. For our purposes, it therefore suffices to consider the subalgebras \(\g(\tau,h) := \g(\Wmin, \tau,h)\) for our classification problem.
The hermitian simple Lie algebras \(\g\) which admit a 3-grading induced by the adjoint representation of an element in \(\g\) are said to be of \emph{tube type}. It is obvious from \eqref{eq:std-wedge-ideal} that \(\g(W,\tau,h) = \g\) is only possible if \(\g\) is of tube type.

The following theorem is the main result of this article:

\begin{thm}
  \label{thm:herm-invclass}
  Let \(\g\) be a hermitian simple Lie algebra.
  Let \(\tau \in \Aut(\g)\) be an involution with \(\tau(\Wmin) = -\Wmin\) and let \(h \in \g\) with \(\tau(h) = h\).
  If the Lie algebra \(\g(\tau,h)\) is non-trivial, then it is either hermitian simple and of tube type or is the direct sum of two such Lie algebras.
  More precisely, \(\g(\tau,h)\) is either \(\{0\}\) or isomorphic to one of the Lie algebras in {\rm Table \ref{table:herm-invclass}}.
  Conversely, for every involutive automorphism \(\tau \in \Aut(\g)\) with \(\tau(\Wmin) = -\Wmin\), all subalgebras \(\g(\tau,h)\) in the table can be realized for some \(h \in \g\) with \(\tau(h) = h\).
  \begin{table}[H]
    \centering
    \begin{tabular}{|c|c|c|c|}
      \hline
      \(\g\) & \(\g^\tau\) & \(\g(\tau,h)\) & \(p,q,k,\ell\)\\
      \hline
      \(\su(p,p)\) & \(\fsl(p,\C) \times \R\) & \(\su(k,k) \oplus \su(\ell,\ell)\) & \(p \geq k + \ell > 0\)\\
      \(\su(p,q)\) & \(\so(p,q)\) & \(\sp(2k,\R)\) & \(p > q \geq k > 0\)\\
      \(\su(2p,2q)\) & \(\sp(p,q)\) & \(\so^*(4k)\), \(\fsl(2,\R)\) & \(p > q > 0, q \geq k > 1\)\\
      \hline
      \(\sp(2p,\R)\) & \(\fsl(p,\R) \times \R\) & \(\sp(2k,\R) \oplus \sp(2\ell,\R)\) & \(p \geq k + \ell > 0\)\\
      \(\sp(4p,\R)\) & \(\sp(2p,\C)\) & \(\su(k,k)\) & \(p \geq k > 0\)\\
      \hline
      \(\so^*(4p)\) & \(\su^*(2p) \times \R\) & \(\so^*(4k) \oplus \so^*(4\ell)\), \(\fsl(2,\R)\), & \(p \geq k + \ell > 1,\, k,\ell \neq 1\)\\
                    & & \(\so^*(4k) \oplus \fsl(2,\R)\) & \\
                    & & \(\fsl(2,\R) \oplus \fsl(2,\R)\) & \\
      \(\so^*(2p)\) & \(\so(p,\C)\) & \(\su(k,k)\) & \(p \geq 3, \lfloor \frac{p}{2} \rfloor \geq k > 0\)\\
      \hline
      \(\so(2,p)\) & \(\so(1,1) \times \so(1,p-1)\) & \(\fsl(2,\R) \oplus \fsl(2,\R), \so(2,k)\) & \(p > 0, p \neq 2, k \in \{p, 1\}\) \\
      \(\so(2,p)\) & \(\so(q,1) \times \so(1,p-q)\) & \(\so(2,k)\) & \(q > 1, k \in \{1,p - q + 1, q + 1\}\)\\
      \(\so(2,p)\) & \(\so(1,p)\) & \(\so(2,1)\) & \\
      \hline
      \(\fe_{6(-14)}\) & \(\sp(2,2)\) & \(\so(2,5), \fsl(2,\R)\) & \\
      \(\fe_{6(-14)}\) & \(\ff_{4(-20)}\) & \(\fsl(2,\R)\) & \\
      \hline
      \(\fe_{7(-25)}\) & \(\fe_{6(-26)} \times \R\) & \(\fe_{7(-25)}, \so(2,10), \fsl(2,\R)\), & \\
                       &  & \(\so(2,10) \oplus \fsl(2,\R)\) & \\
      \(\fe_{7(-25)}\) & \(\su^*(8)\) & \(\so^*(12), \so^*(8), \fsl(2,\R)\) & \\
      \hline
    \end{tabular}
    \caption{Lie subalgebras \(\g(\tau,h)\) for an involution \(\tau \in \Aut(\g)\) with \(\tau(\Wmin) = -\Wmin\) and \(h \in \g^\tau\).}
    \label{table:herm-invclass}
  \end{table}
\end{thm}
Our interest in the subalgebras \(\g(W,\tau,D)\) comes from the theory of standard subspaces: A \emph{standard subspace} of a complex Hilbert space \(\cH\) is a closed real subspace \(V\) such that \(V \cap iV = \{0\}\) and \(V + iV\) is dense in \(\cH\).
We denote the set of standard subspaces of \(\cH\) by \(\Stand(\cH)\).
In Algebraic Quantum Field Theory, standard subspaces can be constructed naturally in the context of Haag--Kastler nets (cf.\ \cite{Ha96}): Here, one studies nets of von Neumann algebras
\[\cO \mapsto \cM(\cO)\]
on a Hilbert space \(\cH\) indexed by open regions in a spacetime \(X\).
One of the properties of such a net is the existence of a unit vector \(\Omega \in \cH\) that is \emph{cyclic} for each \(\cM(\cO)\), i.e.\  \(\cM(\cO)\Omega\) spans a dense subspace of \(\cH\). If \(\Omega\) is also \emph{separating} for some \(\cM(\cO)\), i.e.\ if the map \(\cM(\cO) \ni A \mapsto A\Omega\) is injective, then we obtain a standard subspace by setting
\[V_\cO := \oline{\{A\Omega : A \in \cM(\cO), A^* = A\}}.\]
Conversely, one can pass from standard subspaces to von Neumann algebras via Second Quantization in such a way that locality and isotony properties are preserved.

We refer to \cite{Lo08} for more information about standard subspaces and to \cite{Ar63} and \cite[Sec.\ 6]{NO17} for more information about the translation procedure outlined above.

One particularly interesting problem in this context is to understand the order structure on \(\Stand(\cH)\) defined by the inclusion relation.
As the arguments in \cite{Ne18} show, a natural approach is to examine, for a fixed standard subspace \(V \in \Stand(\cH)\) and a unitary representation \((U,\cH)\) of a Lie group \(G\), the orbit \(U(G)V \subset \Stand(\cH)\).
The order structure on \(U(G)V\) is encoded in the closed subsemigroup
\begin{equation}
  \label{eq:stand-semigrp}
  S_V := \{g \in G : U(g)V \subset V\}.
\end{equation}
Infinitesimally, the subsemigroup \(S_V\) can be described by its \emph{Lie wedge} \(\L(S_V)\), which is a semigroup analog of a Lie algebra (cf.\ \cite{HHL89}):
\[\L(S_V) := \{x \in \L(G) : \exp(\R_{\geq 0} x) \subset S_V\}.\]
Here, \(\L\) denotes the Lie functor.

The Lie wedge \(\L(S_V)\) can be described explicitly in the following setting:
Let \(\tau_G\) be an involutive automorphism of \(G\) and suppose that \((U,\cH)\) extends to a representation of \(G \rtimes \{\1, \tau_G\}\) such that \(U(\tau_G)\) is antiunitary.
Let \(\tau := \L(\tau_G)\). 
For \(h \in \g^\tau,\) we denote the infinitesimal generator of the unitary one-parameter group \(t \mapsto U(\exp(th))\) by \(\partial U(h)\). Then
\[J := U(\tau_G) \quad \text{and} \quad \Delta := e^{2\pi i \partial U(h)}\]
determine a standard subspace \(V := \Fix(J\Delta^{1/2}) \in \Stand(\cH)\), and \(S_V\) is a subsemigroup of \(G\) which is invariant under the operation \(g \mapsto g^* := \tau_G(g)^{-1}, g \in G\).

The following result from \cite{Ne19} then shows that \(\L(S_V)\) can be determined using the structure theory of Lie algebras: Let \(C_U := \{x \in \g : -i\p U(x) \geq 0\}\) be the \emph{positive cone of \(U\)} and suppose that the kernel of \(U\) is discrete, i.e.\ \(C_U\) is pointed. The closed convex cone \(C_U\) is invariant and we have \(\tau(C_U) = -C_U\) because
\[-i\p U(\tau(x)) = -iJ\p U(x) J = J(i\p U(x))J \quad \text{for } x \in \g.\]

Then the Lie wedge \(\L(S_V)\) is given by
\begin{equation}
  \L(S_V) = C_-(C_U, \tau, h) \oplus \g^\tau_0(h) \oplus C_+(C_U, \tau, h).
\end{equation}
and \(\g_\mathrm{red} := \spann(\L(S_V))\) is a 3-graded Lie algebra. On \(\g_\mathrm{red}\), the involution \(\tau\) coincides with the involution \(e^{i\pi \ad h }\) (cf.\ \cite[Thm.\ 4.4]{Ne19}).

When it comes to the order structure on \(U(G)V\) defined by the semigroup \(S_V\) in \eqref{eq:stand-semigrp}, we are mainly interested in strict inclusions of standard subspaces, which in turn correspond to elements which are not invertible in \(S_V\). Hence, it is reasonable to focus on the ideal in \(\g_\mathrm{red}\) that is generated by \(C_+(C_U, \tau, h)\) and \(C_-(C_U, \tau, h)\). If \(\g_C := C_U - C_U\) denotes the ideal of \(\g\) generated by \(C_U\), then this means that \(\g_C(C_U,\tau, \ad h)\) is the 3-graded causal subalgebra we are interested in.

\subsection*{Content of this article}
This article is divided into two main parts: In Section \ref{sec:liealg-conv-jordan}, we provide the background knowledge that is necessary for the proof of Theorem \ref{thm:herm-invclass}.
We first recall some general facts on convex sets in Lie algebras in Section \ref{sec:conv-liealg}.
In Section \ref{sec:herm-liealg}, we introduce the basic structure theory of hermitian simple Lie algebras and examine those involutive automorphisms of hermitian simple Lie algebras that flip the minimal invariant cone \(\Wmin\).
Our proof of Theorem \ref{thm:herm-invclass} relies heavily on the structure theory of nilpotent orbits of convex type and their associated \(\fsl_2\)-triples. We therefore recall all the definitions and results which are needed for the proof of the main theorem from \cite{HNO94} and \cite{Sa80}. These results are based on the work on the structure theory of bounded symmetric domains by Kor{\'a}nyi and Wolf (cf.\ \cite{KW65}, \cite{WK65}).
Hermitian simple Lie algebras are closely related to simple euclidean Jordan algebras via the Kantor--Koecher--Tits construction.
We will explain this relation in more detail in Section \ref{sec:jordan-algebras} because it is crucial for the proof of Theorem \ref{thm:herm-invclass}.

Finally, in Section \ref{sec:liewedge-std-subspace}, we prove Theorem \ref{thm:herm-invclass}. In the above context, the Lie wedge \(\L(S_V)\) generates \(\g\) only if \(\tau = e^{i\pi\ad(h_0)}\) for some semisimple element \(h_0 \in \g\) such that \(\spec(\ad(h_0)) \subset \Z\). We will consider this case in detail in Section \ref{sec:cayley-type-inv} and summarize the results in Theorem \ref{thm:cayley-type-class-simplepart}. 
The main step towards proving the general case is done in Section \ref{sec:reductions} and, more specifically, in Theorem \ref{thm:invclass-5-grad-reduction}, where we reduce the classifications of the subalgebras \(\g(\tau,h)\) to the case where \(h \in \g^\tau\) is hyperbolic and induces a 5-grading on \(\g\).
The first part of Theorem \ref{thm:herm-invclass} is one of the consequences of these results and is shown independently from the actual classification (cf.\ Theorem \ref{thm:invclass-3grad-jordan}). 

\subsection*{Notation}

\begin{itemize}
  \item For a vector space \(V\) over \(\K \in \{\R,\C\}\), a linear endomorphism \(A\) on \(V\), and \(\lambda \in \K\), we denote the \(\lambda\)-eigenspace of \(A\) by \(V(A;\lambda)\).

  \item Let \(\g\) be a Lie algebra and let \(x \in \g\). For an \(\ad(x)\)-invariant subspace \(V \subset \g\), we define \(V_\lambda(x) := V(\ad(x); \lambda)\). For an involutive automorphism \(\tau \in \Aut(\g)\), we define
    \[\g^{\pm \tau} := \{y \in \g : \tau(y) = \pm y\}.\]
    If \(\g\) is real semisimple and \(\theta \in \Aut(\g)\) is a Cartan involution, then we define \(\fk := \g^\theta\) and \(\fp := \g^{-\theta}\).

    If \(\g\) is hermitian simple, then we denote the (up to sign unique) minimal invariant pointed generating closed convex cone by \(\Wmin\) instead of \(\Wmin(\g)\) if there is no ambiguity. Moreover, we denote by \(\Inn(\g) := \la e^{\ad \g} \ra\) the group of inner automorphisms of \(\g\).
\end{itemize}

\section{Lie algebras, convex cones, and Jordan algebras}
\label{sec:liealg-conv-jordan}

In this section, we introduce the structure theory of Lie algebras containing convex cones that is necessary in order to prove Theorem \ref{thm:herm-invclass}. Since it is crucial for our proof, we will also recall the relation between Jordan algebras and hermitian simple Lie algebras.

\subsection{Convexity in Lie algebras}
\label{sec:conv-liealg}

\begin{definition}
  A real finite dimensional Lie algebra \(\g\) is called \emph{admissible} if it contains a generating invariant closed convex subset \(C\) with \(H(C) := \{x \in \g : C + x = C\} = \{0\}\), i.e.\ \(C\) contains no affine lines.
\end{definition}

\begin{definition}
  Let \(\g\) be a Lie algebra and \(\fk \subset \g\) be a subalgebra. Then \(\fk\) is said to be \emph{compactly embedded} if the subgroup generated by \(e^{\ad \fk}\) is relatively compact in \(\Aut(\g)\).
\end{definition}

Every admissible Lie algebra \(\g\) contains a compactly embedded Cartan subalgebra \(\ft\) (cf.\ \cite[Thm.\ VII.3.10]{Ne00}) and there exists a maximal compactly embedded subalgebra \(\fk\) of \(\g\) containing \(\ft\).
Let \(\Delta := \Delta(\g_\C, \ft_\C)\) be the corresponding set of roots.
A root \(\alpha \in \Delta\) is called \emph{compact} if \(\g^\alpha_\C\) is contained in \(\fk_\C\).
Otherwise it is called \emph{non-compact}.
The set of compact roots is denoted by \(\Delta_k\) and the set of non-compact roots by \(\Delta_p\).

Let \(\Delta^+ \subset \Delta\) be a system of positive roots, i.e.\ there exists an element \(X \in i\ft\) such that \(\Delta^+ = \{\alpha \in \Delta : \alpha(X) > 0\}\) and \(\alpha(X) \neq 0\) for all \(\alpha \in \Delta\). Then we say that \(\Delta\) is \emph{adapted} if \(\beta(X) > \alpha(X)\) for all \(\alpha \in \Delta_k, \beta \in \Delta_p^+\).

A Lie algebra \(\g\) which contains a compactly embedded Cartan subalgebra is called \emph{quasihermitian} if there exists an adapted system of positive roots.
An equivalent condition is that \(\fz_\g(\fz(\fk)) = \fk\) for a maximal compactly embedded subalgebra \(\fk\) of \(\g\). Every simple quasihermitian Lie algebra is either compact or satisfies \(\fz(\fk) \neq \{0\}\). In the latter case, \(\g\) is \emph{hermitian}.

\begin{definition}
  \label{def:weyl-group-compact}
  Let \(\g\) be admissible, let \(\fk\) be a compactly embedded subalgebra of \(\g\), and let \(\ft \subset \fk\) be a compactly embedded Cartan subalgebra. Denote by \(\Delta := \Delta(\g_\C,\ft_\C)\) the corresponding set of roots. For every \(\alpha \in \Delta_k\), there exists a unique element \(\alpha^\vee \in [\g_\C^\alpha,\g_\C^{-\alpha}]\) with \(\alpha(\alpha^\vee) = 2\). We call the group \(\cW_\fk\) generated by the reflections
  \[s_\alpha : \ft \rightarrow \ft, \quad X \mapsto X - \alpha(X)\alpha^\vee, \quad \alpha \in \Delta_k,\]
  the \emph{Weyl group of the pair \((\fk, \ft)\)} (cf.\ \cite[Def.\ VII.2.8]{Ne00}).
\end{definition}

Let \(\g\) be admissible, let \(\fk \subset \g\) be a maximal compactly embedded subalgebra, and let \(\ft\) be a compactly embedded Cartan subalgebra. Let \(\Delta := \Delta(\g_\C,\ft_\C)\) be the corresponding set of roots. Then the pointed generating invariant closed convex cones \(W \subset \g\) can be classified as follows (cf. \cite[VII.3, VIII.3]{Ne00}): For a system of positive roots \(\Delta^+ \subset \Delta\), we define the convex cones
\[\Cmin := \Cmin(\Delta_p^+) := \cone\{i[Z_\alpha,\oline{Z_\alpha}] : Z_\alpha \in \g_\C^\alpha, \alpha \in \Delta_p^+\} \subset \ft\]
and
\[\Cmax := \Cmax(\Delta_p^+) := \{X \in \ft : (\forall \alpha \in \Delta_p^+) i\alpha(X) \geq 0\}.\]
  Then there exists a unique adapted system of positive roots such that \(\Cmin \subset W \cap \ft \subset \Cmax\), and \(W \cap \ft\) uniquely determines \(W\).

  Conversely, if \(C \subset \ft\) is a generating closed convex cone with \(\Cmin \subset C \subset \Cmax\) for some adapted system of positive roots which is invariant under the Weyl group \(\cW_\fk\) of \((\fk,\ft)\), then there exists a unique generating invariant closed convex cone \(W\) with \(W \cap \ft = C\) (cf.\ \cite[Thm.\ VIII.3.21]{Ne00}).

\subsection{Hermitian simple Lie algebras}
\label{sec:herm-liealg}

A simple non-compact quasihermitian Lie algebra is called \emph{hermitian}.
One can also define the hermitian simple Lie algebras as those real simple Lie algebras which contain a pointed generating invariant cone (cf.\ \cite[Thm.\ 7.25]{HN93}).

Let \(\g\) be a hermitian simple Lie algebra.
The classification of hermitian simple Lie algebras shows that, for every maximal compactly embedded subalgebra \(\fk \subset \g\), we have \(\dim \fz(\fk) = 1\).
In particular, there exists an element \(H_0 \in \fz(\fk)\) such that \(\ad(H_0)\) is diagonalizable in \(\g_\C\) with \(\spec(\ad H_0) = \{-i,0,i\}\).
We call such an element an \emph{\(H\)-element}.
For an \(H\)-element \(H_0 \in \fz(\fk)\), we define \(\Wmin(H_0)\) as the smallest pointed generating invariant closed convex cone containing \(H_0\), and \(\Wmax(H_0)\) as the negative dual cone \(-\Wmin^*(H_0)\) with respect to the Killing form of \(\g\).
These cones are up to a sign uniquely characterized by \(\Wmin(H_0) \subset \Wmax(H_0)\) and the property that, for every pointed generating invariant closed convex cone \(W \subset \g\), either \(\Wmin(H_0) \subset W \subset \Wmax(H_0)\) or \(\Wmin(H_0) \subset (-W) \subset \Wmax(H_0)\) (cf.\ \cite[Thm.\ 7.25]{HN93}).
Because of their uniqueness up to a sign, we also denote these cones by \(\Wmin(\g)\) and \(\Wmax(\g)\) or simply \(\Wmin\) and \(\Wmax\).

Let \(\g = \fk \oplus \fp\) be a Cartan decomposition of \(\g\) and let \(\fa \subset \fp\) be a maximal abelian subspace.
Let \(\Sigma \subset \fa^*\) be the restricted root system of \(\g\) and let \(r := \dim \fa =: \rk_\R \g\) be the real rank of \(\g\).
Then, according to Moore's Theorem (cf.\ \cite{Mo64}, \cite[Ch.\ III, \S 4]{Sa80}), the restricted root system \(\Sigma\) is either of type \((C_r)\) or of type \((BC_r)\), i.e.\ we either have
\begin{equation}
  \label{eq:tube-type}
  \Sigma = \{\pm(\varepsilon_j \pm \varepsilon_i) : 1 \leq i < j \leq r\} \cup \{\pm 2\varepsilon_j : 1 \leq j \leq r\} \cong C_r
\end{equation}
or
\begin{equation}
  \label{eq:non-tube-type}
\Sigma = \{\pm(\varepsilon_j \pm \varepsilon_i) : 1 \leq i < j \leq r\} \cup \{\pm \varepsilon_j, \pm 2\varepsilon_j : 1 \leq j \leq r\} \cong BC_r.
\end{equation}
We say that \(\g\) is of \emph{tube type} if \(\Sigma \cong C_r\) and that \(\g\) is of \emph{non-tube type} if \(\Sigma \cong BC_r\).
Let \(H_k\) be the coroot of \(2\varepsilon_k\) for \(k=1,\ldots,r\).
Then \(\{H_1,\ldots,H_r\}\) is a basis of \(\fa\) and we have \(\varepsilon_k(H_\ell) = \delta_{k \ell}\) for \(1 \leq \ell \leq r\).

\begin{rem}(The Weyl-group of a hermitian simple Lie algebra)
  \label{rem:herm-weylgrp}
  The Weyl-group \(\cW\) of a hermitian simple Lie algebra \(\g\) can be determined as follows: Fix a Cartan decomposition \(\g = \fk \oplus \fp\) and a maximal abelian subspace \(\fa\). Let \(r := \dim \fa = \rk_\R \g\) and let \(H_1,\ldots,H_r\) be defined as above.

  Let \(\beta\) be the Cartan-Killing form on \(\g\). Then there exists for each \(\alpha \in \fa^*\) a unique element \(A_\alpha \in \fa\) such that \(\beta(H,A_\alpha) = \alpha(H)\) for all \(H \in \fa\). The Weyl group of \(\g\) can be identified with the group generated by the reflections
  \[s_\alpha(H) = H - 2\frac{\alpha(H)}{\alpha(A_\alpha)} A_\alpha, \quad H \in \fa, \alpha \in \Sigma.\]
  (cf.\ \cite[Ch.\ VII, Cor.\ 2.13]{Hel78}). Using the root space decomposition of \(\g\), we find that \(A_{\varepsilon_k}\) is a multiple of \(H_k\) for \(1 \leq k \leq r\). A straightforward computation now shows that \(\cW\) can be identified with the group \(\{\pm 1\}^r \rtimes S_r\) of signed permutations on \(\{1,\ldots,r\}\) in the obvious way.
\end{rem}

\begin{rem}(On hyperbolic elements)
  \label{rem:hyperbolic-elements}
  Let \(\g\) be a real semisimple Lie algebra and let \(G\) be a connected Lie group with \(\L(G) = \g\).
  Recall that an \(\ad\)-diagonalizable element \(x \in \g\) is called \emph{hyperbolic} if \(\spec(\ad x) \subset \R\), and that it is called \emph{elliptic} if \(\spec(\ad x) \subset i\R\).

  Let \(\g = \fk \oplus \fp\) be a Cartan decomposition of \(\g\), let \(\fa \subset \fp\) be a maximal abelian subspace of \(\fp\), and let \(\g = \fk + \fa + \fn\) be the corresponding \emph{Iwasawa decomposition} of \(\g\) (cf.\cite[Ch.\ VI, \S 3]{Hel78}). By \cite[Ch.\ IX, Thm.\ 7.2]{Hel78}, an element \(h \in \g\) is hyperbolic if and only if it is conjugate to an element in \(\fa\), i.e.\ \(\fa \cap \cO_h \neq \emptyset\).
  
  On the Lie group level, we have the decomposition \(G = KAN\). For an element \(h \in \fa\), we are interested in the set \(\cO_h \cap \fa = \Ad(G)h \cap \fa\). From the Iwasawa decomposition of \(G\), we see that this set equals \(\Ad(K)h \cap \fa\), which equals \(\cW.h\), where \(\cW\) is the Weyl group of \(\g\) (cf.\ \cite[Ch.\ VII, \S 2]{Hel78}).
\end{rem}

Apart from the type of the restricted root system, we have the following criterion to distinguish tube type hermitian simple Lie algebras from non-tube type hermitian simple Lie algebras:

\begin{lem}
  \label{lem:tube-type-3grad}
  Let \(\g\) be a hermitian simple Lie algebra. Then \(\g\) is of tube type if and only if there exists an element \(0 \neq h \in \g\) such that \(\g = \g_{-1}(h) \oplus \g_0(h) \oplus \g_1(h)\).
\end{lem}
\begin{proof}
  Fix a Cartan decomposition \(\g = \fk \oplus \fp\) and a maximal abelian subspace \(\fa \subset \fp\). Let \(H_1,\ldots,H_r\) be the coroots of \(2\varepsilon_1,\ldots,2\varepsilon_r\) respectively, where \(r = \rk_\R \g\).
  By Remark \ref{rem:hyperbolic-elements}, every hyperbolic element in \(\g\) is conjugate to an element of the form \(h = \sum_{k=1}^r \lambda_k H_k \in \fa\) for some \(\lambda_1,\ldots,\lambda_r \geq 0\).
  If \(\ad(h)\) induces a 3-grading on \(\g\) in the above form, then it is easy to see that \(\lambda_k = \frac{1}{2}\) for all \(1 \leq k \leq r\).
  But then \(\ad(h)\) induces a 3-grading on \(\g\) if and only if the restricted root system \(\Sigma \subset \fa^*\) is of type \((C_r)\), which proves the claim.
\end{proof}

\begin{rem}
  \label{rem:tube-type-3grad-sign}
  Let \(\g\) be a hermitian simple Lie algebra of tube type and of real rank \(r\). Fix a Cartan decomposition \(\g = \fk \oplus \fp\) of \(\g\) and a maximal abelian subspace \(\fa \subset \fp\). Let \(H_1,\ldots,H_r \in \fa\) be defined as before.

  An element \(h \in \fa\) induces a 3-grading on \(\g\) with \(\g = \g_{-1}(h) \oplus \g_0(h) \oplus \g_1(h)\) if and only if it is of the form \(h = \sum_{k=1}^r \lambda_k H_k\) with \(\lambda_k \in \{\pm \frac{1}{2}\}\).
  Similarly, it is easy to see that \(h\) induces a 5-grading on \(\g\) with \(\g = \g_{-1}(h) \oplus \g_{-\frac{1}{2}}(h) \oplus \g_0(h) \oplus \g_{\frac{1}{2}}(h) \oplus \g_1(h)\) if and only if \(h = \sum_{k=1}^r \lambda_k H_k\) with \(\lambda_k \in \{0, \pm \frac{1}{2}\}\).

  For \(1 \leq k \leq r\), let \(H_k' := \lambda_k H_k\) and define \(\varepsilon_k' := \lambda_k \varepsilon_k\). Thus, by replacing \(H_k\) with \(H_k'\) and \(\varepsilon_k\) with \(\varepsilon_k'\), we may assume that \(\lambda_k\) is non-negative for all \(1 \leq k \leq r\).
\end{rem}

Reductive Lie algebras containing pointed generating invariant convex cones can be decomposed as follows:

\begin{lem}
  \label{lem:adm-red-liealg-ideals}
  Let \(\g\) be an admissible reductive Lie algebra. Then every simple ideal in \(\g\) is either compact or hermitian.
\end{lem}
\begin{proof}
  Every admissible Lie algebra is quasihermitian by \cite[Thm.\ VII.3.10]{Ne00}, so that we have \(\fz_\g(\fz(\fk)) = \fk\) for all maximal compactly embedded subalgebras. Let \(\fk\) be such a maximal compactly embedded subalgebra and let \(\fs\) be a simple ideal. Then \(\fk_\fs := \fk \cap \fs\) is maximal compactly embedded in \(\fs\) and, since \(\g\) is reductive, we have \(\fz_\fs(\fz(\fk_\fs)) = \fk_\fs\), so that \(\fs\) is also quasihermitian. Hence, \(\fs\) is either compact or hermitian simple (cf.\ Section \ref{sec:conv-liealg}).
\end{proof}

\subsubsection{Nilpotent elements}
\label{sec:nilpotent-elements}
Let \(\g\) be a finite-dimensional real Lie algebra. An element \(x \in \g\) is called \emph{nilpotent} if \(\ad x\) is a nilpotent endomorphism of \(\g\).

If \(\g\) is semisimple, then every nilpotent element \(x \neq 0\) can be embedded into an \emph{\(\fsl(2)\)-triple} \((h,x,y) \in \g^3\), i.e.\ we have
\[[h,x] = 2x, \quad [h,y] = -2y, \quad \text{and} \quad  [x,y] = h\]
(cf.\ \cite[Ch.\ IX, Thm.\ 7.4]{Hel78}).

A sufficient condition for \(x \in \g\) being nilpotent is the existence of an element \(h \in \g\) such that \([h,x] \in  \R x\).
To see this, consider the subalgebra \(\g_x := \R h \ltimes \R x\), where \(h\) acts by the adjoint representation.
Then \(\g_x\) is a solvable Lie algebra.
Hence, every representation of \(\g_x\) restricted to \([\g_x,\g_x] = \R x\) is nilpotent (cf.\ \cite[Cor.\ 5.4.11]{HN12}), so that in particular \(x\) is nilpotent.

For our purposes the nilpotent elements which are contained in pointed generating invariant cones are of particular interest:

\begin{definition}
  \label{def:orbit-conv-type}
  Let \(x \in \g\). Then the adjoint orbit \(\cO_x := \Inn(\g)(x)\) is called an \emph{orbit of convex type} if \(\oline{\cone(\cO_x)}\) is pointed.
\end{definition}

\begin{definition}
  \label{def:sl2-elements}
  The theory of \(\fsl(2)\)-triples is a key tool in the classification of nilpotent orbits of convex type in semisimple Lie algebras. Following \cite{HNO94}, we therefore fix the following notation for elements in \(\fsl(2,\R)\):
  \[H := \pmat{1 & 0 \\ 0 & -1}, \quad X := \pmat{0 & 1 \\ 0 & 0}, \quad Y := \pmat{0 & 0 \\ 1 & 0}, \quad T := \pmat{0 & 1 \\ 1 & 0}, \quad U := \pmat{0 & 1 \\ -1 & 0}.\]
  We have the following commutator relations:
  \[[H,X] = 2X, \quad [H,Y] = -2Y, \quad [X,Y] = H, \quad [U,T] = 2H, \quad [U,H] = -2T, \quad [H,T] = 2U.\]
\end{definition}

\begin{definition}
  \label{def:h-element}
  (cf.\ \cite{Sa80}) Let \(\g\) be a semisimple Lie algebra. Then \(H_0 \in \g\) is called an \emph{\(H\)-element} if \(\ker \ad(H_0)\) is a maximal compactly embedded subalgebra of \(\g\) and \(\spec(\ad H_0) = \{0,i,-i\}\). The pair \((\g, H_0)\) is called a \emph{Lie algebra of hermitian type}.
\end{definition}

We already noted at the beginning of Section \ref{sec:herm-liealg} that every hermitian simple Lie algebra contains an \(H\)-element which is unique up to sign for a fixed Cartan involution. In the case of the hermitian simple Lie algebra \(\g = \fsl(2,\R)\) and the Cartan involution \(\theta(X) := -X^\tran, X \in \g\), the element \(H_0 = \frac{1}{2}U\) is an \(H\)-element.

\begin{definition}
  Let \((\g, H_0)\) and \((\tilde \g, \tilde H_0)\) be semisimple Lie algebras of hermitian type.

  (a) A Lie algebra homomorphism \(\kappa : \g \rightarrow \tilde \g\) is called an \((H_1)\)-homomorphism if \(\kappa \circ \ad(H_0) = \ad(\tilde H_0) \circ \kappa\).

  (b) A Lie algebra homomorphism \(\kappa : \g \rightarrow \tilde \g\) is called an \((H_2)\)-homomorphism if \(\kappa(H_0) = \tilde H_0\).
\end{definition}

\begin{rem}
  \label{rem:h1-hom-cartan}
  Let \((\g,H_0)\) and \((\g',H_0')\) be Lie algebras of hermitian type. Then the \(H\)-element \(H_0\) determines a Cartan decomposition of \(\g\) by \(\fk := \ker(\ad H_0)\) and \(\fp := [H_0,\g]\) (cf.\ \cite[Def.\ II.1]{HNO94}). Let \(\fk' := \ker(\ad H_0')\) and \(\fp' := [H_0',\g]\). Then every \((H_1)\)-homomorphism \(\kappa: (\g,H_0) \rightarrow (\g',H_0')\) satisfies \(\kappa(\fk) \subset \fk'\) and \(\kappa(\fp) \subset \fp'\).
\end{rem}

\begin{rem}
  \label{rem:tube-type-h-element}
  (a) The Lie algebra \(\g = \fsl(2,\R)^r, r \in \N\), is of hermitian type. We introduce the following notation: For an element \(A \in \{H,X,Y,T,U\} \subset \fsl(2,\R)\) (cf.\ Definition \ref{def:sl2-elements}), we denote by \(A_k\) the image of \(A\) under the inclusion of \(\fsl(2,\R)\) into the \(k\)-th summand of \(\g\), where \(1 \leq k \leq r\). An \(H\)-element of \(\g\) is thus given by \(\frac{1}{2}U^r := \frac{1}{2}\sum_{k=1}^r U_k\).

  (b) Let \(\g\) be a hermitian simple Lie algebra with a Cartan decomposition \(\g = \fk \oplus \fp\). Let \(H_0 \in \fz(\fk)\) be an \(H\)-element of \(\g\) (Definition \ref{def:h-element}) and define \(J := \ad(H_0)\).

  Choose a maximal abelian subspace \(\fa \subset \fp\) of dimension \(r = \rk_\R \g\).
  Then, according to \cite[Lem.\ III.3]{HNO94}, \(\fs := \fa + J(\fa) + [\fa, J(\fa)]\) is a subalgebra of \(\g\) which is the image of an injective \((H_1)\)-homomorphism \(\kappa_1 : (\fsl(2,\R)^r, \frac{1}{2}U^r) \rightarrow (\g, H_0)\).
  Furthermore, there exists an \((H_1)\)-homomorphism \(\kappa_2: (\fsl(2,\R),\frac{1}{2}U) \rightarrow (\fsl(2,\R)^r,\frac{1}{2}U^r)\) with \(\kappa_2(H) = \sum_{k=1}^r H_k\) (cf.\ \cite[Prop.\ III.7]{HNO94}). Thus, \(\kappa := \kappa_1 \circ \kappa_2\) is also an \((H_1)\)-homomorphism. We set \(H' := \kappa(H)\).
  If we define \(\{\varepsilon_1,\ldots,\varepsilon_r\} \subset \fa^*\) as the dual basis of \(\{\kappa_1(H_1),\ldots,\kappa_1(H_r)\}\), then we obtain a root system of type \((C_r)\) as in \eqref{eq:tube-type} if \(\g\) is of tube type and of type \((BC_r)\) as in \eqref{eq:non-tube-type} if \(\g\) is of non-tube type (cf.\ \cite[p.\ 110]{Sa80}).
  
  The endomorphism \(\ad(H')\) induces a 5-grading
  \[\g = \g_{-2}(H') \oplus \g_{-1}(H') \oplus \g_0(H') \oplus \g_1(H') \oplus \g_2(H')\]
  and we have \(\g_{\pm 1} = \{0\}\) if and only if \(\g\) is of tube type.

  The homomorphism \(\ad_\g \circ \kappa : \fsl(2, \R) \rightarrow \g\) induces the structure of an \(\fsl(2,\R)\)-module on \(\g\).
  This module is completely reducible and each irreducible submodule is a highest weight module of weight \(\lambda\) for some \(\lambda \in \Z_{\geq 0}\), where the highest vector is an eigenvector of \(H\).
  We denote the isotypic component of a weight \(\lambda\) by \(\g^{[\lambda]}\).
  Then the isotypic decomposition of \(\g\) is given by \(\g = \g^{[0]} \oplus \g^{[1]} \oplus \g^{[2]}\) (cf.\ \cite[Ch.\ III \S 1, Lem.\ 1.2]{Sa80}) and the summand \(\g^{[1]}\) vanishes if and only \(\g\) is of tube type.
  The former condition is equivalent to \(\kappa\) being an \((H_2)\)-homomorphism, i.e.\ \(\kappa(\frac{1}{2}U) = H_0\) (cf.\ \cite[Ch.\ III \S 1, Cor.\ 1.6]{Sa80}).
  Thus, we can identify \(H_0\) with the \(H\)-element \(\frac{1}{2} \sum_{k=1}^r U_k\) of the semisimple Lie algebra \(\fsl(2,\R)^r\) if \(\g\) is of tube type.

  (c) The arguments in (b) show that every hermitian simple Lie algebra of non-tube type contains a hermitian simple Lie algebra of tube type with the same real rank as \(\g\): We define
  \[\g_\mathrm{ev} := \g_{-2}(H') \oplus \g_0(H') \oplus \g_2(H') \quad \text{and} \quad \g_t := \g_t(H') := \g_{-2}(H') \oplus [\g_{-2}(H'), \g_2(H')] \oplus \g_2(H')\]
  and claim that \(\g_t\) is a tube-type Lie algebra.
  The 5-gradings of real simple Lie algebras and the corresponding subalgebras \(\g_\mathrm{ev}\) have been classified in \cite{Kan93}.
  If \(\g\) is a non-tube type hermitian simple Lie algebra, we deduce from \cite[Table II]{Kan93} that, for the grading induced by \(\ad H\) as above, the subalgebra \(\g_\mathrm{ev}\) is of the form \(\g_\mathrm{ev}^s \oplus \fz(\g_\mathrm{ev})\) with \(\g_\mathrm{ev}^s\) hermitian simple and of tube type. In particular, we have \(\g_t = \g_\mathrm{ev}^s\). With this procedure, we can determine \(\g_t\) up to isomorphy for each non-tube type hermitian simple Lie algebra \(\g\) (cf.\ Table \ref{table:rem:non-tube-emb}). The existence of a maximal hermitian simple tube type subalgebra of \(\g\) is also shown in \cite{KW65} and \cite[Lem.\ 4.8]{O91}.
  \begin{table}[H]
    \centering
    \begin{tabular}{|c|c|c|c|}
      \hline
      \(\g\) & \(\su(p,q) \, (p > q > 0)\) & \(\so^*(4n+2) \, (n > 1)\) & \(\fe_{6(-14)}\)\\
      \hline
      \(\g_t\) & \(\su(q,q)\) & \(\so^*(4n)\) & \(\so(2,8)\) \\
      \hline
    \end{tabular}
    \caption{Hermitian simple non-tube type Lie algebras \(\g\) and their tube type subalgebras \(\g_t\) with \(\rk_\R \g = \rk_\R \g_t\) induced by 5-gradings on \(\g\).}
    \label{table:rem:non-tube-emb}
  \end{table}

  (d) Since the elements \(\kappa(X) \in \g_2(H')\) and \(\kappa(Y) \in \g_{-2}(H')\) are nilpotent elements of convex type by the proof of \cite[Thm.\ III.9]{HNO94}, they are contained in \((-\Wmin(\g)) \cup \Wmin(\g)\). In particular, the intersection \(\Wmin(\g) \cap \g_t\) is non-trivial, so that we have \(\Wmin(\g_t) \subset \Wmin(\g)\) if the sign of \(\Wmin(\g_t)\) is chosen appropriately.
\end{rem}

\subsubsection{Involutive automorphisms of hermitian simple Lie algebras}
\label{sec:inv-aut-herm}

Throughout this section, unless stated otherwise, let \(\g\) be a hermitian simple Lie algebra and let \(\g = \fk \oplus \fp\) be a Cartan decomposition of \(\g\). In view of Theorem \ref{thm:herm-invclass}, we recall a few basic facts about those involutive automorphisms \(\tau \in \Aut(\g)\) that flip the minimal invariant cone, i.e.\ \(\tau(\Wmin) = -\Wmin\).
The following lemma provides a useful criterion for this property in terms of the \(H\)-elements of \(\g\):

\begin{lemma}
  \label{lem:inthypinv-cone-helement}
  Let \(\tau \in \Aut(\g)\) be an involution that preserves \(\fk\) and \(\fp\) and let \(H_0 \in \fz(\fk)\) be an \(H\)-element of \(\g\). Then \(\tau(H_0) \in \{\pm H_0\}\). Moreover, we have \(\tau(\Wmin(H_0)) = -\Wmin(H_0)\) if and only if \(\tau(H_0) = -H_0\).
\end{lemma}
\begin{proof}
  Let \(G\) be a connected Lie group with Lie algebra \(\g\) and let \(K \subset G\) be a maximal compact subgroup with Lie algebra \(\fk\). Then it is well known (cf.\ e.g.\ \cite{Ja75}, \cite{HO97}) that \(\tau\) defines an involution on the bounded symmetric domain \(G/K\), which is either holomorphic if \(\tau(H_0) = H_0\), or antiholomorphic if \(\tau(H_0) = -H_0\).
  This proves the first part.

  If \(\tau(\Wmin(H_0)) = -\Wmin(H_0)\), then \(\tau(H_0) = -H_0\) because \(H_0 \in \Wmin(H_0)\). Conversely, if \(\tau(H_0) = -H_0\), then \(-\Wmin(H_0) \subset \tau(\Wmin(H_0))\) because \(-\Wmin(H_0)\) is the minimal invariant cone containing \(-H_0\). Applying \(\tau\) on both sides yields the converse inclusion.
\end{proof}

\begin{definition}
  \label{def:auto-equiv}
  For a Lie algebra \(\g\), let \(\tau,\sigma \in \Aut(\g)\) be involutive automorphisms of \(\g\). We say that \(\tau\) and \(\sigma\) are \emph{equivalent} if there exists an automorphism \(\varphi \in \Aut(\g)\) such that \(\tau = \varphi^{-1} \circ \sigma \circ \varphi\).
\end{definition}

\begin{rem}
  \label{rem:comp-causal-herm}
  (a) Let \(\tau \in \Aut(\g)\) be an involution preserving \(\fk\) and \(\fp\) with \(\tau(\Wmin) = -\Wmin\), i.e.\ \(\tau(H_0) = -H_0\) (cf.\ Lemma \ref{lem:inthypinv-cone-helement}). Recall that \(\fz(\fk) = \R H_0\). Denote by \(\Wmin^o\) the interior of \(\Wmin\) in \(\g\). Then, by \cite[Prop.\ 7.10]{HN93}, we have \(\Wmin^o \cap \fz(\fk) \neq \emptyset\). Since \(\Wmin\) is invariant under \(-\tau\), this means that the cone \(C := \Wmin \cap \g^{-\tau}\) has inner points because \(C^o = \Wmin^o \cap \g^{-\tau}\) (cf.\ \cite[Prop.\ 1.6]{HN93}). In particular, \(C^o \cap \fk \neq \emptyset\). Hence, \((\g,\tau)\) is a \emph{compactly causal symmetric pair}, i.e.\ there exists a cone \(C \subset \g^{-\tau}\) which is invariant under \(\Inn(\g^\tau)\) such that \(C^o \cap \fk \neq \emptyset\) (cf.\ \cite[Def.\ 3.1.8]{HO97}).

  Conversely, suppose that \((\g,\tau)\) is compactly causal and that \(\g\) is hermitian simple. Then, by \cite[Prop.\ 3.1.12]{HO97}, there exists an \(H\)-element \(H_0 \in \g^{-\tau} \cap \fk\), so that \(\tau(\Wmin) = -\Wmin\). Thus, an involutive automorphism \(\tau \in \Aut(\g)\) flips the minimal invariant cone \(\Wmin\) if and only if \((\g,\tau)\) is compactly causal.

  (b) The involutive automorphisms of real simple Lie algebras have been classified by Berger in \cite{B57}. A list of all compactly causal symmetric pairs up to equivalence can be found in \cite[Thm.\ 3.2.8]{HO97}.
\end{rem}

\begin{definition}
  \label{def:inv-cayley-type}
  Let \(\tau \in \Aut(\g)\) be an involution with \(\tau(\Wmin) = -\Wmin\).
  Let \(\g^c := \g^\tau \oplus i\g^{-\tau}\) be the corresponding dual Lie algebra. Then we say that \((\g, \tau)\) is of \emph{Cayley type} if \((\g, \tau) \cong (\g^c, \tau)\).
\end{definition}

\begin{definition}
  \label{def:integral-hyperbolic}
  Let \(\g\) be a Lie algebra. An element \(h \in \g\) is called \emph{integral hyperbolic} if \(\ad(h)\) is diagonalizable with \(\spec(\ad h) \subset \Z\).
\end{definition}

\begin{prop}
  \label{prop:inv-cayley-type}
  \begin{enumerate}
    \item Let \(0 \neq h \in \g\) such that \(\g = \g_{-1}(h) \oplus \g_0(h) \oplus \g_1(h)\) and set \(\tau_h := e^{i\pi \ad h}\). Then \((\g, \tau_h)\) is of Cayley type.
    \item Let \(h \in \g\) be an integral hyperbolic element such that \(e^{i\pi \ad h}(\Wmin) = -\Wmin\). Then \((\g, \tau_h)\) is of Cayley type.
    \item Let \(\tau \in \Aut(\g)\) be an involution such that \((\g, \tau)\) is of Cayley type. Then the following holds:
      \begin{enumerate}
        \item \(\g\) is of tube type.
        \item For every \(h \in \g\) with the properties from {\rm (a)}, the involutions \(\tau\) and \(\tau_h\) are equivalent.
      \end{enumerate}
  \end{enumerate}
\end{prop}
\begin{proof}
  (a) The symmetric pair \((\g, \tau)\) is parahermitian by \cite[Prop.\ 4.1]{KK85}. Thus, \cite[Lem.\ 2.6]{O91} and \cite[Thm.\ 2.7]{O91} imply that \((\g, \tau)\) is also compactly causal. The claim now follows from \cite[Thm.\ 5.6]{O91}.

  For the proof of (b) and (c), we also refer to \cite[Thm.\ 5.6]{O91}.
\end{proof}

Let \(\tau \in \Aut(\g)\) be an involutive automorphism which leaves \(\fk\) and \(\fp\) invariant. We have seen in Lemma \ref{lem:inthypinv-cone-helement} that \(\tau(\Wmin) \in \{\pm \Wmin\}\) because \(\tau(H_0) \in \{\pm H_0\}\), where \(H_0 \in \fz(\fk)\) is an \(H\)-element of \(\g\). This statement is false for general pointed generating invariant convex cones, as the following example shows.
However, in the context of standard subspaces which we outlined in the introduction, the condition \(\tau(W) = -W\) arises naturally from the construction of the standard subspace.

\begin{example}
  For \(n \in \N, n > 1\), consider the Lie algebra
  \[\g := \su(n,n) := \left\{ \pmat{\alpha & \beta \\ \beta^* & \gamma} \in \gl(2n, \C) : \alpha^* = -\alpha, \gamma^* = -\gamma, \tr(\alpha) + \tr(\gamma) = 0 \right\}.\]
  Then \(\g\) is a hermitian simple Lie algebra of tube type with \(\g_\C = \fsl(2n,\C)\). As a Cartan involution, we fix \(\theta(X) := -X^*, X \in \g\). Then
  \[\ft := \left\{ \pmat{\alpha & 0 \\ 0 & \gamma} \in \g : \alpha = \diag(ia_1,\ldots,ia_n), \gamma = \diag(ic_1,\ldots,ic_n), a_1,c_1,\ldots,a_n,c_n \in \R\right\} \subset \fk\]
  is a compactly embedded Cartan subalgebra. Let \(X = \pmat{\alpha & 0 \\ 0 & \gamma} \in \ft\). The compact Weyl group \(\cW_\fk\) of \((\fk,\ft)\) (cf.\ Definition \ref{def:weyl-group-compact}) is generated by the permutations of the diagonal entries of \(\alpha\) and the permutations of diagonal entries of \(\gamma\). The minimal, respectively maximal, invariant cone is uniquely determined by the cone
  \[\Cmin := \left\{\pmat{\alpha & 0 \\ 0 & \gamma} \in \ft : a_1,\ldots,a_n \geq 0 \geq c_1,\ldots,c_n\right\},\]
  respectively
  \[\Cmax := \left\{\pmat{\alpha & 0 \\ 0 & \gamma} \in \ft : a_1,\ldots,a_n \geq c_1,\ldots,c_n\right\}\]
  (cf.\ \cite[p. 331]{Pa81}). The \(\cW_\fk\)-invariant convex cone
  \[C := \left\{\pmat{\alpha & 0 \\ 0 & \gamma} \in \ft : a_1,\ldots,a_n \geq 0,\, a_1,\ldots,a_n \geq c_1,\ldots,c_n\right\}\]
  satisfies \(\Cmin \subset C \subset \Cmax\), so that there exists a pointed generating invariant closed convex cone \(W \subset \g\) with \(W \cap \ft = C\) (cf.\ Section \ref{sec:conv-liealg}).

  Consider the involutive automorphism
  \[\tau : \g \rightarrow \g, \quad X \mapsto \pmat{0 & -\bbone_n \\ \bbone_n & 0}X\pmat{0 & \bbone_n \\ -\bbone_n & 0}.\]
  Then \(\tau(\Wmin) = -\Wmin\) because the \(H\)-element \(H_0 = \pmat{i\bbone_n & 0 \\ 0 & -i\bbone_n} \in \fz(\fk)\) is mapped to \(-H_0\). Moreover, we have \(\g^\tau \cong \fsl(n,\C) \times \R\) (cf.\ \cite[Ex.\ 5.10]{O91}). But \(\tau(W) \neq -W\) because, for any element \(X = \pmat{\alpha & 0 \\ 0 & \gamma} \in C\) for which \(\gamma\) has a diagonal entry with a strictly positive imaginary part, we have
  \[-\tau(X) = \pmat{-\gamma & 0 \\ 0 & -\alpha} \not\in C.\]
\end{example}

\subsection{Jordan algebras}
\label{sec:jordan-algebras}

In this section, we recall some basic facts about Jordan algebras and their relation to hermitian simple Lie algebras.

\begin{definition}
  \label{def:jordan-algebra}
  (a) Let \(V\) be a vector space over \(\K \in \{\R,\C\}\) and let
  \[\cdot : V \times V \rightarrow V, \quad (x,y) \mapsto xy := x \cdot y\]
  be a bilinear map. Then \((V,\cdot)\) is called a \emph{Jordan algebra} if
  \[x \cdot y = y \cdot x \quad \text{and} \quad x \cdot (x^2 \cdot y) = x^2 \cdot (x \cdot y)\]
  for all \(x,y \in V\).

  (b) A real Jordan algebra \(V\) is called \emph{euclidean} if there exists a scalar product \(\la \cdot, \cdot \ra\) on \(V\) which is associative in the sense that \(\la x \cdot y, z \ra = \la y, x \cdot z \ra \text{ for all } x,y,z \in V.\)

  (c) A Jordan algebra is called \emph{simple} if it does not contain any non-trivial ideal.

  (d) Let \(V\) be a Jordan algebra. We denote the left multiplication by an element \(x \in V\) by \(L(x)\) and define the \emph{quadratic representation of \(V\)} by \(P(x) := 2L(x)^2 - L(x^2)\).

  (e) For a Jordan algebra \(V\) over \(\K\), \(x \in V\), and \(\lambda \in \K\), we define \(V_\lambda(x) := V(L(x); \lambda)\).
\end{definition}

\begin{example}
  \label{ex:KKT}
  (a) Let \((\g,\sigma)\) be a real symmetric Lie algebra and suppose that there exists a hyperbolic element \(h \in \g\) with \(\spec(\ad h) = \{-2,0,2\}\) and elements \(x,y \in \g\) such that \((h,x,y)\) is an \(\fsl_2\)-triple with
    \[\sigma h = -h, \quad \sigma x = -y, \quad \sigma y = -x.\]
    Then \(\ad h\) induces a 3-grading \(\g = \g_{-2}(h) \oplus \g_0(h) \oplus \g_2(h)\) such that \(\sigma(\g_{\pm 2}(h)) = \g_{\mp 2}(h)\). Furthermore, we can endow \(V := \g_2(h)\) with a Jordan algebra structure via
    \[a \cdot b := -\frac{1}{2} [[a, \sigma x], b] = \frac{1}{2} [[a,y],b] \quad \text{for } a,b \in V.\]
    This is also known as the \emph{Kantor--Koecher--Tits} construction (cf.\ \cite{FK94}, \cite[Ch.\ 1 \S 2.7]{Koe69}). The element \(x \in V\) is a \emph{unit element} with respect to the Jordan algebra product, i.e.\
    \[x \cdot b = -\frac{1}{2} [[x,\sigma x], b] = \frac{1}{2} [h,b] = b \quad \text{for } b \in V.\]

    (b) In particular, if \(\g\) is real semisimple and \(\sigma = \theta\) is a Cartan involution, then \(V\) as above becomes a euclidean Jordan algebra. Indeed, if \(\beta\) denotes the Cartan--Killing form of \(\g\), then \(\la x, y \ra := -\beta(x,\theta y),\, x,y \in V,\) is an associative scalar product on \(V\) (cf.\ \cite[Ch.\ II, \S 5]{Koe69}). Furthermore, \(V\) is simple if and only if \(\g\) is simple (cf. \cite[Thm.\ 1]{Koe67}).
\end{example}

For a euclidean Jordan algebra \(V\), we define \(\Omega_V := \{x^2 : x \in V\}^o\) as the interior of the set of squares of \(V\). It is self-dual in the sense that \(\Omega_V = \{x \in V : (\forall y \in V)\, \la x,y^2 \ra \geq 0\}^o\) (cf.\ \cite[Thm.\ III.2.1]{FK94}), and is a homogeneous cone, i.e.\ the automorphism group
\[G(\Omega_V) := \{g \in \GL(V) : g\Omega_V = \Omega_V\}\]
of \(\Omega_V\) acts transitively on \(\Omega_V\).

\begin{rem}
  \label{rem:tkk-onetoone}
  Let \((\g,\theta)\) be as in Example \ref{ex:KKT}(b).

  (a) The Kantor--Koecher--Tits construction from Example \ref{ex:KKT} establishes a one-to-one correspondence between the isomorphy classes of hermitian simple Lie algebras of tube type and the isomorphy classes of real simple euclidean Jordan algebras: As shown in \cite{Koe67} (see also \cite[Ch.\ 1 \S 7]{Sa80}), one can construct from a semisimple euclidean Jordan algebra \(V\) a semisimple 3-graded Lie algebra \(\g = \g_{-1} \oplus \g_0 \oplus \g_1\) with \(\g_1 = V\) and a Cartan involution \(\theta\) on \(\g\) such that \(\theta(\g_{\pm 1}) = \g_{\mp 1}\) and \(\theta(\g_0) = \g_0\).
  We refer to \cite[p.\ 213]{FK94} for a complete list of this one-to-one correspondence. 

  (b) \(\g_0\) is \(\theta\)-invariant, so that we obtain an eigenspace decomposition \(\g_0 = \g_0^\theta \oplus \g_0^{-\theta}\). This direct sum can be identified with the \emph{structure algebra} \(\str(V) := \Der(V) \oplus L(V)\) of \(V\), where
  \[\Der(V) := \{A \in \End(V) : (\forall x,y \in V)\, D(x \cdot y) = Dx \cdot y + x \cdot Dy\}\]
  is a Lie subalgebra and \(L(V)\) is closed under triple brackets (cf.\ \cite[Ch.\ I, \S 7]{Sa80}).
\end{rem}

\begin{rem}
  \label{rem:KKT-cone-liealg}
  In the context of Example \ref{ex:KKT}, the Lie algebra \(\g(\Omega_V) := \L(G(\Omega_V))\) of the automorphism group of \(\Omega_V\) is isomorphic to \(\g_0 = \ker(\ad h)\), which acts on \(V\) by inner Lie algebra automorphisms (cf.\ \cite[Ch.\ 1, \S 8]{Sa80}).
\end{rem}

\begin{definition}
  Let \(V\) be a real or complex Jordan algebra with a unit \(e\).

  (a) An element \(c \in V\) is called an \emph{idempotent} if \(c^2 = c\).

  (b) A subset \(\{c_1, \ldots, c_r\} \subset V\) is called a \emph{complete system of orthogonal idempotents} if
  \[c_i^2 = c_i, \quad c_ic_j = 0, \quad c_1 + \ldots + c_r = e \quad \text{for } i,j \in \{1,\ldots,r\}, i \neq j.\]

  (c) A complete system of orthogonal idempotents \(\{c_1,\ldots,c_r\}\) is called a \emph{Jordan frame} if each \(c_i\) (\(i \in \{1,\ldots,r\}\)) is non-zero and cannot be written as the sum of two non-zero orthogonal idempotents.

  (d) Let \(F := \{c_1,\ldots,c_r\}\) be a Jordan Frame of \(V\). We call \(r = |F|\) the \emph{rank} of \(V\). It does not depend on the choice of the Jordan frame (cf.\ \cite[Cor.\ IV.2.7]{FK94}).
\end{definition}

\begin{example}
  \label{ex:herm-jordan-frame}
  Let \(\g\) be a hermitian simple tube type Lie algebra of real rank \(r\) with a Cartan involution \(\theta\). Fix an \(H\)-element \(H_0 \in \fz(\fk)\). Recall from Remark \ref{rem:tube-type-h-element}(b) that there exists an injective \((H_1)\)-homomorphism \(\kappa: (\fsl(2,\R)^r, \frac{1}{2}\sum_{k=1}^r U_k) \rightarrow (\g,H_0)\) such that, if we identify \(\kappa(\fsl(2,\R)^r)\) with \(\fsl(2,\R)^r\), the element \(H := \sum_{k=1}^r H_k\) induces a 3-grading on \(\g\) with \(\g = \g_{-2}(H) \oplus \g_0(H) \oplus \g_2(H)\). The space \(\fa := \sum_{k=1}^r \R H_k \subset \fp\) is maximal abelian and the restricted root system \(\Sigma \subset \fa^*\) is of type \((C_r)\) as in \eqref{eq:tube-type} if we define \(\{\varepsilon_1,\ldots,\varepsilon_r\}\) as the dual basis of \(\{H_1,\ldots,H_r\}\).

  Endow \(V := \g_2(H)\) with the structure of a euclidean simple Jordan algebra as in Example \ref{ex:KKT} such that \(e := \sum_{k=1}^r X_k\) is a unit element.
  By Remark \ref{rem:h1-hom-cartan}, we have \(\theta(e) = -\sum_{k=1}^r Y_k\).
  For \(1 \leq k \leq r\), the element \(X_k \in \g^{2\varepsilon_k} \subset V\) is an idempotent since
  \[X_k^2 = -\tfrac{1}{2}[[X_k,\theta(e)],X_k] = \tfrac{1}{2}[[X_k, Y_k], X_k] = \tfrac{1}{2}[H_k,X_k] = X_k.\]
  Each \(X_k\) is also a primitive idempotent: Since \(V\) is euclidean, \(X_k\) is primitive if and only if \(\dim V_1(X_k) = 1\) by the Spectral theorem \cite[Thm.\ III.1.1]{FK94}. But this follows from
  \[X_k \cdot a = -\tfrac{1}{2} [[X_k, \theta(e)], a] = \tfrac{1}{2} [H_k, a],\quad a \in V,\]
  \cite[Lem.\ IV.7]{HNO94}, and \(\dim \g_2(H_k) = 1\).
  Hence, the set \(\{X_1,\ldots,X_r\}\) is a Jordan frame of \(V\).
\end{example}

\begin{lem}
  \label{lem:jordan-herm-cone}
  Let \(\g\) be a hermitian simple Lie algebra of tube type with a Cartan decomposition \(\g = \fk \oplus \fp\) and let \(h \in \fp\) be a hyperbolic element with \(\spec(\ad h) = \{-2,0,2\}\). Let \(\g = \g_{-2}(h) \oplus \g_0(h) \oplus \g_2(h)\) be the 3-grading induced by \(\ad h\) and endow \(V := \g_2(h)\) with the structure of a simple euclidean Jordan algebra as in {\rm Example \ref{ex:KKT}}. Let \(\Wmin \subset \Wmax \subset \g\) be the minimal, respectively maximal, invariant proper cones in \(\g\). Then \(\Wmin \cap \g_2(h) = \Wmax \cap \g_2(h) \in \{\pm \oline{\Omega_V}\}\).
\end{lem}
\begin{proof}
  Let \(r := \rk_\R \g\) be the real rank of \(\g\).
  We first recall that \(\g_2(h)\) consists of nilpotent elements (cf.\ Section \ref{sec:nilpotent-elements}). In particular, \(\Wmax \cap \g_2(h)\) consists of nilpotent elements of convex type. Hence, we have \(\Wmax \cap \g_2(h) \subset \Wmin \cap \g_2(h)\) by \cite[Thm.\ III.9]{HNO94}.
  Let \(X_1,\ldots,X_r\) be defined as in Example \ref{ex:herm-jordan-frame}. By the proof of \cite[Thm.\ III.9]{HNO94}, we may without loss of generality assume that \(X_1,\ldots,X_r \in \Wmin\).
  For \(1 \leq k \leq r\), let \(X(k) := \sum_{\ell=1}^k X_\ell\).
  Then, since \(\{X_1,\ldots,X_r\}\) is a Jordan frame, we have by \cite[Prop.\ IV.3.1]{FK94} that \(\oline{\Omega_V} = \bigcup_{k=1}^r G(\Omega_V)_0 X(k)\). Combined with Remark \ref{rem:KKT-cone-liealg}, this shows that \(\oline{\Omega_V} \subset \Wmin \cap \g_2(h)\).
  Conversely, let \(x \in \Wmin \cap \g_2(h)\). Then the adjoint orbit \(\cO_x\) is nilpotent and of convex type. Hence, we have \(x \in \oline{\Omega_V} \cup (-\oline{\Omega_V})\) by \cite[Prop.\ V.9]{HNO94}.
\end{proof}

We also recall the \emph{Peirce decomposition} of a Jordan algebra \(V\): For an idempotent \(c \in V\), the vector space \(V\) decomposes into the direct sum \(V = V_0(c) \oplus V_{\frac{1}{2}}(c) \oplus V_1(c)\). Given a Jordan frame \(\{c_1,\ldots,c_r\}\), we thus obtain a decomposition
\begin{equation}
  \label{eq:peirce-decomp-frame}
  V = \bigoplus_{i=1}^r V_i \oplus \bigoplus_{1 \leq i < j \leq r} V_{ij}, \quad \text{where } V_i := \R c_i \text{ and } V_{ij} := V_{\frac{1}{2}}(c_i) \cap V_{\frac{1}{2}}(c_j) = V_{ji} ~~ (i \neq j).
\end{equation}
Moreover, the following multiplication rule holds (cf.\ \cite[Thm.\ IV.2.1]{FK94}):
\begin{equation}
  \label{eq:peirce-mult}
  V_{ij} \cdot V_{k\ell} \subset \begin{cases} V_i + V_j & \{i,j\} = \{k,\ell\}, \\ V_{i\ell} & j = k, i \neq \ell, \\ \{0\} & \{i,j\} \cap \{k,\ell\} = \emptyset.\end{cases}
\end{equation}

\begin{rem}
  \label{rem:herm-frame-roots}
  (a) In the setting of Example \ref{ex:herm-jordan-frame}, we have the following relation between the Peirce decomposition with respect to the Jordan frame \(\{X_1,\ldots,X_r\}\) and the root space decomposition of the Lie algebra \(\g\): Fix a maximal abelian subspace \(\fa \subset \fp\). Recall that the restricted root system \(\Sigma \subset \fa^*\) is of type \((C_r)\) as in \eqref{eq:tube-type}. The elements \(H_k\) are coroots of \(2\varepsilon_k\) for \(k=1,\ldots,r\).
Thus, a straightforward computation using \cite[Lem.\ IV.7]{HNO94} shows that
\begin{equation}
  \label{eq:herm-frame-roots}
  V_i = \g^{2\varepsilon_i} \quad \text{and} \quad V_{ij} = \g^{\varepsilon_i + \varepsilon_j} \quad \text{for } i,j \in \{1,\ldots,r\}, i \neq j.
\end{equation}

(b) Let \(V\) be a real semisimple euclidean Jordan algebra and let \(F := \{c_1,\ldots,c_r\}\) be a Jordan frame of \(V\). Recall from Remark \ref{rem:tkk-onetoone} that we obtain from the Kantor--Koecher--Tits construction a 3-graded Lie algebra \(\g = \g_{-1} \oplus \g_0 \oplus \g_1\) with \(\g_1 = V\) and \(\g_0 = \str(V)\) and a Cartan involution \(\theta\) of \(\g\) with \(\theta(\g_{\pm 1}) = \g_{\mp 1}\) and \(\theta(\g_0) = \g_0\). Then the decomposition \(\str(V) = \Der(V) \oplus L(V)\) is the Cartan decomposition of \(\str(V)\) corresponding to \(\theta\lvert_{\g_0}\) (cf.\ \cite[Ch.\ 1 \S 8]{Sa80}). Moreover, \(\fa := \spann(L(F)) \subset L(V)\) is maximal abelian in \(L(V)\).

Actually, \(\fa\) is maximal abelian in \(\fp = \g^{-\theta}\). One can show that the restricted root system \(\Sigma \subset \fa^*\) is of type \((C_r)\) as in \eqref{eq:tube-type} if one defines the linear functionals \(\varepsilon_k \in \fa^*\) by \(\varepsilon_k(L(c_\ell)) := \frac{1}{2}\delta_{k \ell}\) for \(1 \leq k,\ell \leq r\) (cf.\ \cite[p.\ 212]{FK94}).
\end{rem}

\begin{lem}
  \label{lem:squares-peirce}
  Let \(V\) be a euclidean Jordan algebra with unit element \(e\) and a Jordan frame \(\{c_1,\ldots,c_r\}\). Let \(y \in \{x^2 : x \in V\}\) and let \(y = \sum_{i=1}^r \lambda_i c_i + \sum_{1 \leq i < j \leq r} y_{ij}\) be the Peirce decomposition of \(y\). If \(\lambda_i = 0\) for some \(i \in \{1,\ldots,r\}\), then \(y_{\ell i} = y_{i j} = 0\) for all \(\ell < i < j\).
\end{lem}
\begin{proof}
  Let \(x \in V\) such that \(y = x^2\) and let \(x = \sum_{i=1}^r \mu_i c_i + \sum_{1 < i < j \leq r} x_{ij}\) be the Peirce decomposition of \(x\).
  For \(k \in \{1,\ldots,r\}\), let \(p_k : V \rightarrow V_k\) be the projection onto \(V_k\) along \eqref{eq:peirce-decomp-frame}. We compute \(p_k(x^2)\) in terms of the Peirce decomposition of \(x\): According to \cite[Prop.\ IV.1.4]{FK94}, we have \(x_{ij}^2 = \frac{1}{2}\|x_{ij}\|^2(c_i + c_j)\) for \(1 \leq i < j \leq r\). Using \eqref{eq:peirce-mult}, we see that
  \[p_k(y) = \lambda_k c_k = p_k(x^2) = \mu_k^2 c_k + \frac{1}{2} \left( \sum_{j=k+1}^r \|x_{kj}\|^2 + \sum_{j=1}^{k-1} \|x_{jk}\|^2 \right) c_k.\]
  Hence, if \(\lambda_k = 0\), then the right hand side of the above equation also vanishes, which proves the claim because of the multiplication rules in \eqref{eq:peirce-mult}.
\end{proof}

\begin{rem}
  Lemma \ref{lem:squares-peirce} generalizes the fact that, if the \(i\)-th diagonal entry of a positive semidefinite matrix vanishes, then all entries in the \(i\)-th row and column also vanish.

  Moreover, one can even show that \(\|y_{ij}\|^2 \leq 2y_i y_j\) for every square \(y = \sum_{i=1}^r y_i + \sum_{1 \leq i < j \leq r} y_{ij} \in V\) and \(1 \leq i < j \leq r\) (cf.\ \cite[p.\ 80]{FK94}).
\end{rem}

The Peirce decomposition allows us to realize real simple euclidean Jordan algebras as subalgebras of real simple euclidean Jordan algebras of the same family with a larger rank. This can be achieved as follows: For a real euclidean Jordan algebra \(V\) of rank \(r\), we fix a Jordan frame \(\{c_1,\ldots,c_r\}\) and set \(e_j = c_1 + \ldots + c_j\) for \(1 \leq j < r\). Then \(V^{(j)} := V_1(e_j)\) is a Jordan subalgebra of \(V\). Obviously, the cone \(\Omega^{(j)} = \oline{\Omega_{V^{(j)}}}\) of squares in \(V^{(j)}\) is given by \(\oline{\Omega_V} \cap V^{(j)}.\) From the Peirce decomposition \(V = \bigoplus_{i=1}^r V_i \oplus \bigoplus_{1 \leq k < \ell \leq r} V_{k \ell}\), we see that
\begin{equation}
  \label{eq:jordan-simple-subalg-peirce}
  V^{(j)} = \bigoplus_{i=1}^j V_i \oplus \bigoplus_{1 \leq k < \ell \leq j} V_{k \ell}.
\end{equation}

\begin{prop}
  \label{prop:jordan-simple-subalg}
  Let \(V\) be a real simple euclidean Jordan algebra of rank \(r\) and let \(\{c_1,\ldots,c_r\}\) be a Jordan frame of \(V\). Then the subalgebra \(V^{(j)} := V_1(c_1 + \ldots + c_j)\) is simple for all \(1 \leq j \leq r\). Moreover, an idempotent \(c \in V^{(j)}\) is primitive in \(V^{(j)}\) if and only if it is primitive in \(V\) and \(c \bot \{c_k : j < k \leq r\}\).
\end{prop}
\begin{proof}
  cf.\ \cite[Prop.\ 4.1]{GT11}
\end{proof}

\begin{ex}
  \label{ex:jordan-simple-subalg-hermoctonions}
  We consider the euclidean simple Jordan algebra \(V = \Herm(3,\bbO)\) of hermitian matrices over the octonions \(\bbO\). According to the classification of euclidean simple Jordan algebras, the rank of \(V\) is \(3\) and \(\dim V = 27\) (cf.\ \cite[p. 97]{FK94}). We fix a Jordan frame \(\{c_1,c_2,c_3\}\). Since we know by Proposition \ref{prop:jordan-simple-subalg} that \(V^{(1)}\) and \(V^{(2)}\) are simple, we can determine these subalgebras with the same classification and the Peirce decomposition \eqref{eq:jordan-simple-subalg-peirce} of \(V^{(j)}\): First, we note that
  \[\dim V_{ij} = \dim V_{k \ell} \quad \text{for } i,j,k,\ell \in \{1,2,3\}, i \neq j, k \neq \ell\]
  by \cite[Cor.\ IV.2.6]{FK94}, so that \(\dim V_{ij} = 8\). For \(j = 1\), we have \(\rk V^{(1)} = 1\) and \(\dim V = \dim V_1 = 1\), so that \(V^{(1)} \cong \Sym(1,\R) = \R\). For \(j = 2\), we have \(\rk V^{(2)} = 2\). By \cite[Cor.\ IV.1.5]{FK94}, every such Jordan algebra is isomorphic to a Jordan algebra of the form \(\R \times \R^{n-1}\), where the Jordan product is given by
  \[(\lambda,v)(\mu,w) := (\lambda\mu + \la v, w\ra, \lambda v + \mu w) \quad (\lambda,\mu \in \R, v,w \in \R^{n-1}),\]
  and \(\la \cdot, \cdot \ra\) is a positive definite bilinear form on \(\R^{n-1}\). Since \(\dim V^{(2)} = 10\), we have \(V^{(2)} \cong \R \times \R^{10-1}\).
\end{ex}

\subsubsection{Involutive automorphisms of real simple euclidean Jordan algebras}
\label{sec:inv-jordan}

Throughout this section, let \(V\) be a real simple euclidean Jordan algebra with unit element \(e\).
Let \(\alpha \in \Aut(V)\) be an involutive automorphism on \(V\).
Denote the corresponding eigenspace decomposition by \(V = V^\alpha \oplus V^{-\alpha}\).
Then \(V^\alpha\) is a euclidean Jordan subalgebra of \(V\) and we have the inclusions
\[V^\alpha \cdot V^{-\alpha} \subset V^{-\alpha} \quad \text{and} \quad V^{-\alpha} \cdot V^{-\alpha} \subset V^\alpha.\]

Based on the results in \cite{BH98}, we will use the following definition to characterize involutions on \(V\):

\begin{definition}
  \label{def:inv-jordan}
  Let \(V\) be a real simple euclidean Jordan algebra with unit element \(e\) and let \(\alpha \in \Aut(V)\) be an involutive automorphism of \(V\).
  \begin{enumerate}
    \item We say that \(\alpha\) is \emph{split} if \(\rk V^\alpha = \rk V\) and that it is \emph{non-split} if \(2\rk V^\alpha = \rk V\).
    \item We say that \(\alpha\) is a \emph{Peirce reflection} if there exists an idempotent \(c \in V\) such that \(\alpha = P(2c-e)\).
  \end{enumerate}
\end{definition}

\begin{example}
  \label{ex:inv-jordan-minkowski}
  Let \(n \in \N\). We endow the Minkowski space \(\cM^{n+1} := \R \times \R^n\) with the Jordan algebra structure
  \[(x,v) \cdot (y,w) := (xy + \la v,w \ra, xw + yv), \quad x,y \in \R, v,w \in \R^n,\]
  where \(\la v,w \ra := \sum_{k=1}^n v_k w_k\) denotes the standard scalar product on \(\R^n\). Then \(\cM^{n+1}\) is a real simple Jordan algebra of rank \(2\). It is euclidean because the standard scalar product on \(\R^{n+1}\) is associative in the sense of Definition \ref{def:jordan-algebra}. The element \(e := (1,0)\) is a unit element of \(\cM^{n+1}\).

  The automorphism group of \(\cM^{n+1}\) can be described as follows: Let \(\alpha\) be an automorphism and let \(W := \{0\} \times \R^n\). Then \(\alpha(W) = W\): To see this, we only have to observe that \(W\) is the orthogonal complement of \(e\) with respect to the standard scalar product on \(\R^{n+1}\), and the definition of the Jordan algebra product shows that \(\alpha\) preserves this subspace.
  In particular, \(\tilde\alpha := \alpha\lvert_W\) is a surjective linear isometry on \(W\), i.e.\ \(\tilde\alpha \in \OO_n(\R)\).
  Conversely, every surjective linear isometry on \(\R^n\) can be uniquely extended to an automorphism of \(\cM^{n+1}\), so that \(\Aut(\cM^{n+1}) \cong \OO_n(\R)\). Note that the extended automorphism is contained in the indefinite orthogonal group \(\OO_{1,n}(\R)\).

  Consider an involutive orthogonal map \(\alpha \in \OO_{1,n}(\R)\) with \(\alpha(e) = e\) and its restriction \(\tilde\alpha := \alpha\lvert_W\).
  Since \(\tilde\alpha\) is also involutive, we obtain an eigenspace decomposition \(W = W(\tilde\alpha; 1) \oplus W(\tilde\alpha; -1)\) of \(W\).
  The Jordan subalgebra \((\cM^{n+1})^\alpha = \R \times W(\tilde\alpha; 1)\) is of rank \(2\) if and only if \(W(\tilde\alpha; 1)\) is non-trivial.
  Let \(k := \dim W(\tilde\alpha;1)\).
  Then \((\cM^{n+1})^\alpha \cong \cM^{k+1}\).
  In particular, \(\alpha\) is non-split if and only if \(k=0\), i.e.\ \(\tilde\alpha = -\id_{\R^n}\), and it is a Peirce reflection if and only if \(k=1\), in which case the fixed point subalgebra is isomorphic to the non-simple Jordan algebra \(\R \times \R\).
\end{example}

\begin{prop}
  \label{prop:class-inv-jordan}
  Let \(V\) be a real simple euclidean Jordan algebra of rank \(r\) and let \(\alpha \in \Aut(V)\) be an involutive automorphism. Then \(\alpha\) is either split or non-split. Moreover, the following assertions hold:
  \begin{enumerate}
    \item \(V^\alpha\) is either simple or a direct sum \(V^{(k)} \oplus V^{(\ell)}\) for \(0 \leq k,\ell \leq r\) with \(k + \ell = r\) and for some Jordan frame \(F\). The latter case occurs if and only if \(\alpha\) is a Peirce reflection.
    \item There exists a Jordan frame \(F\) of \(V\) such that \(\alpha(F) = F\) and, for \(R := \spann(F)\), we have \(R = R^\alpha\) if \(\alpha\) is split and \(\dim R^\alpha = \dim R^{-\alpha}\) if \(\alpha\) is non-split. The rank of \(V^\alpha\) equals \(\dim R^\alpha\).
  \end{enumerate}

\end{prop}
\begin{proof}
  That \(\alpha\) is either split or non-split is a consequence of the classification in \cite[Table 1.5.1]{BH98}.
  For the proof of statement (a), we refer to \cite[Satz 2.3]{Hw69} (see also \cite[Rem.\ 1.5.3]{BH98}).
  Statement (b) is shown in \cite[Thm.\ 1.6.1]{BH98}.
\end{proof}

\begin{rem}
  \label{rem:class-inv-jordan-frame-conv}
  Note that, in the context of Proposition \ref{prop:class-inv-jordan}(b), the existence of a Jordan frame \(F\) with \(\alpha(F) = F\) and no fixed points under \(\alpha\) does not imply that \(\alpha\) is non-split: Consider for \(n \in \N\) the Jordan algebra \(V := \cM^{n+1}\) from Example \ref{ex:inv-jordan-minkowski}.
  Let \(\alpha \in O_{1,n}(\R)\) be involutive with \(\alpha(1,0) = (1,0)\) and recall that this implies that \(\alpha\) is an automorphism of the Jordan algebra \(V\).
  Suppose that the restriction \(\tilde\alpha\lvert_{\{0\} \times \R^n}\) is not \(\pm \id_{\R^n}\).
  Then \(\alpha\) is a split involution.
  For every unit vector \(v \in \R^n \setminus \{0\}\), the set \(F(v) := \{(\frac{1}{2}, \frac{1}{2}v), (\frac{1}{2}, -\frac{1}{2}v)\}\) is a Jordan frame of \(V\).
  Choose an eigenvector \(v \in \R^n\) of \(\alpha\) with \(\tilde\alpha(v) = -v\).
  Then \(F(v)\) satisfies \(\alpha(F(v)) = F(v)\) and contains no fixed points under \(\alpha\).
\end{rem}

\begin{rem}
  \label{rem:class-inv-jordan}
  (a) We refer to \cite[Table 1.5.1]{BH98} for an explicit classification of all involutive automorphisms of simple euclidean Jordan algebras.

  (b) Let \(V\) be a real simple euclidean Jordan algebra, let \(\alpha \in \Aut(V)\) be an involutive automorphism, and let \(F = \{c_1,\ldots,c_r\} \subset V\) be a Jordan frame as in Proposition \ref{prop:class-inv-jordan}(b).
  If \(\alpha\) is split, then we have \(\alpha(c_k) = c_k\) for \(1 \leq k \leq r\).
  On the other hand, if \(\alpha\) is non-split and \(r = \rk V = 2s\), then the proof of \cite[Thm.\ 1.6.1]{BH98} shows that, after renumbering the elements of \(F\) if necessary, we have \(\alpha(c_k) = c_{k+s}\) for \(1 \leq k \leq s\). In particular, all \(\alpha\)-invariant subsets of \(F\) have an even number of elements.
  
  Consider the Peirce decomposition \(V = \left(\bigoplus_{k=1}^r \R c_k\right) \oplus \bigoplus_{1 \leq i < j \leq r} V_{ij}\) of \(V\) with respect to \(F\). If \(\alpha\) is split, then the subspaces \(V_{ij}\) are invariant under \(\alpha\), and if \(\alpha\) is non-split, the direct sums
  \[V_{ij} \oplus V_{(i + s),(j + s)} \oplus V_{(i+s),j} \oplus V_{i,(j+s)} \quad \text{for } 1 \leq i < j \leq s\]
  are \(\alpha\)-invariant (cf.\ \cite[1.7.1]{BH98}). Consequently, the subalgebras \(V^{(s)} := V_1(c_1 + \ldots + c_k)\) for \(1 \leq k \leq r\) are \(\alpha\)-invariant in the split case and the subalgebras \(V_1(c_1 + c_{1 + s} + \ldots + c_k + c_{k+s}) \cong V^{(2k)}\) for \(1 \leq k \leq s\) are \(\alpha\)-invariant in the non-split case.
\end{rem}

\begin{lem}
  \label{lem:inv-jordan-fixedpts-rk}
  Let \(V\) be a real simple euclidean Jordan algebra and let \(\alpha \in \Aut(V)\) be an involutive automorphism of \(V\). Let \(c \in V^\alpha\) be a primitive idempotent of \(V^\alpha\) and let \(W := V_1(c)\). Then \(\rk W = 1\) if \(\alpha\) is split and \(\rk W = 2\) if \(\alpha\) is non-split.
\end{lem}
\begin{proof}
  It is shown in the proof of \cite[Thm.\ 1.6.1]{BH98} that the Jordan subalgebras \(V_1(c)\) for idempotents \(c \in V^\alpha\) which are primitive in \(V^\alpha\) are conjugate under automorphisms of \(V\) and are either of rank \(1\) or of rank \(2\). Hence, it suffices to show the claim for one particular \(c \in V^\alpha\) as above.

  Let \(F = \{d_1,\ldots,d_r\}\) be a Jordan frame of \(V\) with the properties from Proposition \ref{prop:class-inv-jordan}(b). If \(\alpha\) is split, then \(d_1 \in V^\alpha\) and \(V_1(d_1) = \R d_1\) because \(d_1\) is primitive in \(V\), so that \(\rk V_1(d_1) = 1\). Obviously, \(d_1\) is also primitive in \(V^\alpha\).

  Conversely, suppose that \(\rk V_1(c) = 1\), i.e.\ \(V_1(c) = \R c\), for every primitive idempotent \(c \in V^\alpha\). Then every primitive idempotent \(c \in V^\alpha\) is also primitive in \(V\), which implies \(\rk V^\alpha = \rk V\), so that \(\alpha\) is split.
\end{proof}

The following lemma is a converse to Remark \ref{rem:class-inv-jordan}(b):

\begin{lem}
  \label{lem:jordan-nonsplit-inv-subalg}
  Let \(V\) be a real simple euclidean Jordan algebra and let \(\alpha \in \Aut(V)\) be a non-split involutive automorphism of \(V\). Let \(F := \{c_1,\ldots,c_{r}\}\) be a Jordan frame of \(V\) with \(\alpha(F) = F\). If the subalgebra \(V^{(k)} := V_1(c_1 + \ldots + c_k)\) is \(\alpha\)-invariant for some \(1 \leq k \leq r\), then \(k\) is even.
\end{lem}
\begin{proof}
  Suppose that \(k = \rk V^{(k)}\) is odd. Then \(\alpha\) leaves the subset \(\{c_1,\ldots,c_{k}\}\) invariant, so that it must have a fixed point in this subset. Suppose that \(c_\ell\) is such a fixed point for some \(1 \leq \ell \leq k\). Then \(c_\ell \in V^\alpha\) with \(\rk V_1(c) = 1\), which contradicts Lemma \ref{lem:inv-jordan-fixedpts-rk}.
\end{proof}

\begin{lem}
  \label{lem:jordan-nonsplit-frame}
  Let \(V\) be a real simple euclidean Jordan algebra of rank \(r = 2s\) and let \(\alpha \in \Aut(V)\) be a non-split involutive automorphism of \(V\). Let \(F := \{c_1,\ldots,c_r\}\) be a Jordan frame of \(V\) such that \(\alpha(F) = F\) and \(\alpha(\{c_1,\ldots,c_s\}) \cap \{c_1,\ldots,c_s\} = \emptyset\). Then \(F_+ := \{c_k + \alpha(c_k) : 1 \leq k \leq s\}\) is a Jordan frame of \(V^\alpha\).
\end{lem}
\begin{proof}
  Let \(1 \leq k \leq s\) and \(d_k := c_k + \alpha(c_k)\). Then \(d_k\) is an idempotent because \(\alpha(c_k) \neq c_k\). Moreover, \(F_+\) consists of pairwise orthogonal elements.
  Let \(W := V_1(d_k) = V_1(c_k + \alpha(c_k))\).
  Then \(W\) is a simple, \(\alpha\)-invariant subalgebra of rank \(2\) and \(\{c_k, \alpha(c_k)\}\) is a Jordan frame of \(W\) (cf.\ Proposition \ref{prop:jordan-simple-subalg}).
  We claim that \(\alpha_W := \alpha\lvert_W\) is non-split.
  Suppose that \(\alpha_W\) was split.
  Then, by Proposition \ref{prop:class-inv-jordan}, there would exist a Jordan frame \(S = \{e_1,e_2\}\) of \(W\) such that \(\alpha_W(e_i) = e_i\) for \(i=1,2\).
  The elements \(e_1,e_2\) are primitive in \(W\), hence primitive in \(V\) (cf.\ Proposition \ref{prop:jordan-simple-subalg}), and \(F' := \{e_1, e_2\} \cup (F \setminus \{c_k,\alpha(c_k)\})\) is an \(\alpha\)-invariant Jordan frame of \(V\).
  Since \(e_1 \in V^\alpha\), the subalgebra \(V_1(e_1)\) is \(\alpha\)-invariant.
  But this is a contradiction because \(\alpha\) is non-split (cf.\ Lemma \ref{lem:jordan-nonsplit-inv-subalg}).
  Thus, \(\alpha_W\) is non-split.
  Since \(W\) is of rank \(2\), this implies that \(W^\alpha\) is one-dimensional (cf.\ Example \ref{ex:inv-jordan-minkowski} and \cite[Table 1.5.1]{BH98}).
  Hence, we have \(W^\alpha = V_1^\alpha(d_k) = \R d_k\), which means that \(d_k\) is primitive in \(V^\alpha\).
  Since \(k\) was arbitrary, this shows that \(F_+\) is a Jordan frame of \(V^\alpha\).
\end{proof}

\section{The Lie wedge of the endomorphism semigroup of a standard subspace}
\label{sec:liewedge-std-subspace}

Let \(\g\) be a real Lie algebra, let \(W \subset \g\) be a pointed generating invariant closed convex cone, let \(\tau \in \Aut(\g)\) be an involution with \(\tau(W) = -W\), and let \(h \in \g^\tau\). For the convenience of the reader we repeat the definition of \(\g(W,\tau,h)\) from the introduction: We define
\[C_\pm(W,\tau,h) := \pm W \cap \g^{-\tau}_{\pm 1}(h), \quad \g_\pm := \g_\pm(W,\tau,h) := C_\pm(W,\tau,h) - C_\pm(W,\tau,h),\]
\[\g_\mathrm{red}(W,\tau,h) := \g_- \oplus \g^\tau_0(h) \oplus \g_+, \quad \text{and} \quad \g(W,\tau,h) := \g_- \oplus [\g_-,\g_+] \oplus \g_+.\]
The following lemma shows that the subspaces \(\g_\mathrm{red}(W,\tau,h)\) and \(\g(W,\tau,h)\) are actually Lie subalgebras:

\begin{lem}
  \label{lem:gred-subalg}
  Let \((\g,\tau)\) be a symmetric Lie algebra, let \(h \in \g^\tau\), and let \(W \subset \g\) be a pointed invariant closed convex cone. Then \(\g_\mathrm{red}(W,\tau,h)\) is a 3-graded Lie algebra.
\end{lem}
\begin{proof}
  (cf.\ \cite[Prop.\ 4.3]{Ne19}) It is easy to see that \([\g_-,\g_+] \subset \g^\tau_0(h)\) and \([\g^\tau_0(h),\g_\pm] \subset \g_\pm\). Hence, it remains to show that \(\g_\pm\) is abelian. Consider the nilpotent subalgebra \(\ff_\pm := \sum_{\lambda > 0} \g_{\pm \lambda}(h)\). Then \(\fn_\pm := (W \cap \ff_\pm) - (W \cap \ff_\pm)\) is an admissible nilpotent Lie algebra, hence is abelian by \cite[Ex.\ VII.3.21]{Ne00}. Since \(\g_\pm \subset \fn_\pm\), this proves the claim.
\end{proof}

\begin{lem}
  \label{lem:assoc-wedge-simple}
  Let \(\g\) be a hermitian simple Lie algebra. For a pointed generating invariant closed convex cone \(W \subset \g\) and \(h \in \g\), let \(W_\pm(h) := (\pm W) \cap \g_{\pm 1}(h)\). Then \(W_\pm(h)\) either equals \((\Wmin)_\pm(h)\) or \(-(\Wmin)_\pm(h)\).
\end{lem}
\begin{proof}
  Without loss of generality, we may assume that \(\Wmin \subset W\) (cf.\ \cite[Thm.\ 7.25]{HN93}). Then it remains to show that \(W_\pm(h) \subset (\Wmin)_\pm(h)\). Since \(W_\pm(h) \subset \g_{\pm 1}(h)\), the cone \(W_\pm(h)\) consists of nilpotent elements (cf.\ Section \ref{sec:nilpotent-elements}) and, since \(W\) is pointed, \(W_\pm(h)\) consists of nilpotent elements of convex type. By \cite[Thm.\ III.9]{HNO94}, all nilpotent elements of convex type are contained in the set \(\Wmin \cup (-\Wmin)\), which proves the claim.
\end{proof}

In view of Lemma \ref{lem:assoc-wedge-simple}, we define, for a hermitian simple Lie algebra \(\g\), an involutive automorphism \(\tau \in \Aut(\g)\) with \(\tau(\Wmin) = -\Wmin\), and \(h \in \g^\tau\),
\[C_\pm(\tau,h) := C_\pm(\Wmin,\tau,h), \quad \g_\mathrm{red}(\tau,h) := \g_\mathrm{red}(\Wmin,\tau,h), \quad \text{and} \quad \g(\tau,h) := \g(\Wmin,\tau,h).\]

\begin{lem}
  \label{lem:std-wedge-hypelements}
  Let \(\g\) be a real semisimple Lie algebra, let \(\theta \in \Aut(\g)\) be a Cartan involution, and let \(\tau \in \Aut(\g)\) be an involutive automorphism commuting with \(\theta\).
  Let \(\fa \subset \fp\) be a maximal abelian subspace in \(\fp\) such that \(\fa_\fh := \fa \cap \g^\tau\) is maximal abelian in \(\g^{\tau, -\theta}\).
  Then every element \(h \in \g^\tau\) which is hyperbolic in \(\g\) is conjugate under \(\Inn(\g^\tau)\) to an element in \(\fa \cap \g^\tau\).
\end{lem}
\begin{proof}
  By conjugating with an element in \(\Inn(\g^\tau)\), we may assume that \(h \in \fp \cap \g^\tau\). Now the claim follows from \cite[Lem. 7]{Ma79} and the subsequent remark in the reference.
\end{proof}

\begin{lem}
  \label{lem:invclass-herm-conj-isom}
  Let \(\g\) be a real Lie algebra, let \(W \subset \g\) be an invariant closed convex cone, and let \(\tau \in \Aut(\g)\) be an involution with \(\tau(W) = -W\).
  Let \(h_1,h_2 \in \g^\tau\) and suppose that \(\varphi(h_1) = h_2\) for some \(\varphi \in \Inn(\g^\tau)\).
  Then \(\varphi(\g(W,\tau,h_1)) = \g(W,\tau,h_2)\).
  In particular, \(\g(W,\tau,h_1)\) and \(\g(W,\tau,h_2)\) are isomorphic.
\end{lem}
\begin{proof}
  Recall that \(\Inn(\g^\tau)\) is generated by elements of the form \(e^{\ad h}\) for \(h \in \g^\tau\).
  Hence, \(\tau\) and \(\varphi\) commute and \(\varphi(\g^{-\tau}) = \g^{-\tau}\). For all \(\lambda \in \R\), we have \(\varphi(\g_\lambda(h_1)) = \g_\lambda(h_2)\). Since \(W\) is invariant under \(\Inn(\g^\tau) \subset \Inn(\g)\), this shows that
  \[\varphi(\g^{-\tau} \cap \g_{\pm 1}(h_1) \cap (\pm W)) = \g^{-\tau} \cap \g_{\pm 1}(h_2) \cap (\pm W),\]
  so that \(\varphi(\g(W,\tau,h_1)) = \g(W,\tau,h_2)\).
\end{proof}

In view of Lemma \ref{lem:invclass-herm-conj-isom} and Lemma \ref{lem:std-wedge-hypelements}, it suffices to consider those subalgebras \(\g(\tau,h)\) for our classification problem for which \(h \in \fp \cap \g^{\tau} =: \g^{\tau,-\theta}\), and \(\theta\) is a Cartan involution on \(\g\) commuting with \(\tau\). We therefore define
\begin{equation}
  \label{eq:invclass-def-aux-set}
  \cA(\g) := \{(\tau,\theta) \in \Aut(\g)^2 : \tau^2 = \id_\g, \tau(\Wmin) = -\Wmin, \theta \text{ Cartan involution}, \tau\theta = \theta\tau\}
\end{equation}
for every hermitian simple Lie algebra \(\g\).

Throughout the following sections, \(\g\) is a hermitian simple Lie algebra, unless stated otherwise.
It is a well known fact (see for example \cite[p.\ 192]{Hel78}) that, for every involutive automorphism \(\tau\) of a real semisimple Lie algebra \(\g\), there exists a Cartan involution \(\theta\) of \(\g\) commuting with \(\tau\).

\subsection{Cayley type spaces and integral hyperbolic elements}
\label{sec:cayley-type-inv}

For our purposes, the following class of Cayley type involutions is particularly interesting: Consider a hyperbolic element \(h_0 \in \g\) such that \(\ad h_0\) induces a 3-grading of the form \(\g = \g_{-1}(h_0) \oplus \g_0(h_0) \oplus \g_1(h_0)\).
Then we obtain a Cayley type involution \(\tau_{h_0}\) by setting \(\g^{\tau_{h_0}} := \g_0(h_0)\) and \(\g^{-\tau_{h_0}} := \g_{-1}(h_0) \oplus \g_1(h_0)\).
We will first look at the situation where \(h \in \g^{\tau_{h_0}}\) is such that \(\g_1(h) \subset \g_1(h_0)\) and determine \(\g(\tau_{h_0},h)\).

Throughout this section, we suppose that \(\g\) is of tube type and of real rank \(r\). Choose a Cartan decomposition \(\g = \fk \oplus \fp\) and a maximal abelian subspace \(\fa \subset \fp\). Then the root system \(\Sigma \subset \fa^*\) is of type \((C_r)\) as in \eqref{eq:tube-type}. We denote by \(H_k\) the coroot of \(2\varepsilon_k\), where \(k=1,\ldots,r\), so that \(\{H_1,\ldots,H_r\}\) is a basis of \(\fa\).

We also fix pointed generating invariant cones \(\Wmin \subset \Wmax \subset \g\) and consider the integral hyperbolic element \(H = \frac{1}{2}\sum_{k=1}^r H_k\) for which the involution \(\tau := e^{i\pi \ad H}\) satisfies \(\tau(\Wmin) = -\Wmin\) (cf.\ Proposition \ref{prop:inv-cayley-type}).
Then \(\ad H\) induces a 3-grading \(\g = \g_{-1}(H) \oplus \g_0(H) \oplus \g_1(H)\) and the eigenspace decomposition of \(\g\) with respect to \(\tau\) is given by \(\g^\tau = \g_0(H)\) and \(\g^{-\tau} = \g_{-1}(H) \oplus \g_1(H)\).
We recall from Lemma \ref{lem:jordan-herm-cone} that \(V := \g_1(H)\) carries the structure of a simple euclidean Jordan such that \(\Wmin \cap \g_1(H) = \Wmax \cap \g_1(H)\) coincides, up to sign, with the set of squares \(C := \{x^2 : x \in V\}\) in \(V\), so that, without loss of generality, we have \(\Wmin \cap \g_1(H) = C\).  Moreover, there exists a Jordan frame \(\{X_1,\ldots,X_r\}\) of \(V\) such that \(\g^{2\varepsilon_k} = V_k\) and \(\g^{\varepsilon_i + \varepsilon_j} = V_{ij}\) for \(k \in \{1,\ldots,r\}, 1 \leq i < j \leq r\) (Remark \ref{rem:herm-frame-roots}).

For every \(h \in \g^\tau\) with \(\g_1(h) \subset \g_1(H)\), we first examine the cones
\[C_\pm (h) := C_\pm(\tau,h) := \g^{-\tau}_{\pm 1}(h) \cap \pm\Wmin.\]
Since only the hyperbolic part of \(h\) contributes to \(C_\pm(h)\), it suffices to consider elements in \(h \in \fa \subset \g^\tau\). We may thus assume that \(h\) is of the form \(h = \sum_{k=1}^r \lambda_k H_k\) with \(\lambda_1,\ldots,\lambda_r \in \R\).

\begin{prop}
  \label{prop:cayley-class-reduction}
  Let \(h = \sum_{k=1}^r \lambda_k H_k \in \fa \subset \g^\tau\) with \(\g_1(h) \subset \g_1(H) = \g_1(\frac{1}{2}\sum_{k=1}^r H_k)\) and let \(h' = \sum_{k=1}^r \lambda_k' H_k\) with
  \[\lambda_k' = \begin{cases} \frac{1}{2} & \lambda_k = \frac{1}{2} \\ 0 & \lambda_k \neq \frac{1}{2}\end{cases} \quad \text{for } k=1,\ldots,r.\]
  Then \(C_\pm(h) = C_\pm(h').\)
\end{prop}
\begin{proof}
  The assumption \(\g_1(h) \subset \g_1(H)\) implies that \(\g^{-\tau}_{\pm 1}(h),\g^{-\tau}_{\pm 1}(h') \subset \g_{\pm 1}(H)\), so that \(C_\pm(h) = \g_1(h) \cap C\) and \(C_\pm(h') = \g_1(h') \cap C\) by the arguments from the beginning of this section. From the root space decomposition of \(\g\) with respect to \(\fa\), we see that \(C_\pm(h') \subset C_\pm(h)\). It suffices to show that \(C_+(h') = C_+(h)\) because the Cartan involution with respect to the decomposition \(\g = \fk \oplus \fp\) interchanges \(C_+(h)\) and \(C_-(h)\), respectively \(C_+(h')\) and \(C_-(h')\).

  Let \(x \in C_+(h)\) and let \(x = \sum_{i=1}^r x_i + \sum_{1 \leq i < j \leq r} x_{ij}\) be the Peirce decomposition of \(x\) with respect to the Jordan frame \(\{X_1,\ldots,X_r\}\) we fixed before.
  Since \(x_i \in \g^{2\varepsilon_i}\) and \(x_{ij} \in \g^{\varepsilon_i + \varepsilon_j}\), we have \(x_i \neq 0\) only if \(2\lambda_i = 1\), i.e.\ \(\lambda_i = \frac{1}{2}\), and \(x_{ij} \neq 0\) only if \(\lambda_i + \lambda_j = 1\) for \(i,j=1,\ldots,r\).
  Thus, if \(\lambda_i \neq \frac{1}{2}\), then \(x_i = 0\), and thus \(x_{ij} = x_{ki} = 0\) for all \(1 \leq k < i < j \leq r\) by Lemma \ref{lem:squares-peirce} because \(x\) is a square. This implies \(x \in \g_1(h')\), so that \(x \in \g_1(h') \cap C = C_+(h')\) and therefore \(C_+(h) \subset C_+(h')\).
\end{proof}

Proposition \ref{prop:cayley-class-reduction} shows that, in order to classify the Jordan algebras \(\g_\pm(\tau,h):= C_\pm(h) - C_\pm(h)\) in the case \(\g_1(h) \subset \g_1(H)\), it suffices to consider elements of the form \(h = h_s := \frac{1}{2}\sum_{k=1}^s H_k\) with \(1 \leq s \leq r\). For this purpose, we recall the definition of the simple Jordan subalgebras \(V^{(s)} := V_1(X_1 + \ldots + X_s) \subset V\) and their homogeneous symmetric cones \(C^{(s)} := C \cap V^{(s)}\) from Section \ref{sec:jordan-algebras}.

\begin{lem}
  \label{lem:caley-class-easy-case}
  Let \(h = \frac{1}{2}\sum_{k=1}^s H_k\) for \(1 \leq s \leq r\). Then
  \[C_+(h) = C^{(s)} \quad \text{and} \quad \g_+(\tau,h) = V^{(s)} = V_1(X_1 + \ldots + X_s),\]
  which is a simple Jordan algebra.
\end{lem}
\begin{proof}
  The root space decomposition of \(\g\) with respect to \(\Sigma\) and the Peirce decomposition \eqref{eq:jordan-simple-subalg-peirce} imply that \(\g_1(h) = V^{(s)}\), hence \(C_+(h) = \g_1(h) \cap C = C^{(s)}\). The simplicity of \(\g_+(\tau,h) = C^{(s)} - C^{(s)}\) follows from Proposition \ref{prop:jordan-simple-subalg}.
\end{proof}

The simple euclidean Jordan algebras \(V^{(s)}\) \((1 \leq s \leq r)\) correspond via the Kantor--Koecher--Tits construction (cf.\ Example \ref{ex:KKT}) to hermitian simple Lie algebras.
By Lemma \ref{lem:caley-class-easy-case}, these hermitian simple Lie algebras are exactly the ideals \(\g(\tau,h) \subset \g_\mathrm{red}(\tau,h)\) generated by \(\g^{-\tau}_{-1}(h) \oplus \g^{-\tau}_1(h)\) for a hyperbolic element \(h \in \g^\tau\). We can determine these hermitian simple Lie algebras by reading off the dimension of the root spaces contained in \(V^{(s)}\) (cf.\ Remark \ref{rem:herm-frame-roots} and \eqref{eq:jordan-simple-subalg-peirce}) and comparing them to the classification of real simple Lie algebras in \cite[p.\ 532--534]{Hel78}.

\begin{thm}
  \label{thm:cayley-type-class-simplepart}
  Let \(h_0 \in \g\) be a hyperbolic element such that \(\g = \g_{-1}(h_0) \oplus \g_0(h_0) \oplus \g_1(h_0)\). Let \(\tau := e^{i\pi \ad h_0}\).
  Then, for all \(h \in \g^\tau\) such that \(\g_1(h) \subset \g_1(h_0)\), the subalgebra \(\g(\tau,h)\) is either \(\{0\}\) or of the following form, where \(1 \leq k \leq r\):
\end{thm}
\begin{table}[H]
  \centering
  \begin{tabular}{|c|c|c|c|}
    \hline
    \(\g\) & \(r\) & \(\g_+(\tau,h)\) & \(\g(\tau,h)\) \\
    \hline
    \(\su(n,n)\) & \(n\) & \(\Herm(k,\C)\) & \(\su(k,k)\)\\
    \(\sp(2n,\R)\) & \(n\) &  \(\Sym(k,\R)\) & \(\sp(2k,\R)\) \\
    \(\so^*(4n)\) & \(n\) & \(\Herm(k,\H)\) &  \(\so^*(4k) (k \neq 1), \fsl(2,\R)\)\\
    \(\fe_{7(-25)}\) & \(3\) & \(\Herm(3,\mathbb{O}), \R \times \R^{10-1},\R\) & \(\fe_{7(-25)}, \so(2,10), \fsl(2,\R)\) \\
    \(\so(2,n)\) & \(2\) & \(\cM^n\) & \(\so(2,n), \so(2,1)\) \\
    \hline
  \end{tabular}
  \caption{Lie subalgebras \(\g(\tau,h)\) for a given involution \(\tau \in \Aut(\g)\) and \(h \in \g^\tau\) with \(\tau = e^{i\pi\ad(h_0)}\) and \(\g_1(h) \subset \g_1(h_0)\).}
  \label{table:cayley-type-class-simplepart}
\end{table}
\begin{proof}
  Fix a Cartan involution \(\theta\) and a maximal abelian subspace \(\fa \subset \fp\) of dimension \(r\) such that \(h,h_0 \in \fa\).
  We endow \(V := \g_1(h_0)\) with the structure of a euclidean simple Jordan algebra as in Example \ref{ex:KKT}.
  The restricted root system \(\Sigma \subset \fa^*\) of \(\g\) is of type \((C_r)\) as in \eqref{eq:tube-type}.
  Let \(H_1,\ldots,H_r \in \fa\) be the coroots of \(2\varepsilon_1,\ldots,2\varepsilon_r\) respectively.
  Since \(h_0\) induces a 3-grading on \(\g\), we may assume that \(h_0 = \frac{1}{2} \sum_{k=1}^r H_k\) (cf.\ Remark \ref{rem:tube-type-3grad-sign}).
  Furthermore, Proposition \ref{prop:cayley-class-reduction} shows that it suffices to consider the case where \(h\) is of the form \(h = \frac{1}{2}\sum_{k=1}^s H_k\) for some \(1 \leq s \leq r\). According to Lemma \ref{lem:caley-class-easy-case}, we then have \(\g_+(\tau,h) = \spann(\Wmin \cap \g^{-\tau}_1(h)) = V^{(s)}\) for a suitable Jordan frame in \(V\).

  By applying the KKT-construction to the euclidean simple Jordan algebra \(V^{(s)}\) (cf.\ Proposition \ref{prop:jordan-simple-subalg}), we obtain the subalgebra of \(\g\) generated by the subspaces \(\g_+(\tau,h)\) and \(\theta(\g_+(\tau,h)) = \g_-(\tau,h)\), which is \(\g(\tau,h)\).
  Using the classification of hermitian simple Lie algebras and the classification of euclidean Jordan algebras (cf.\ \cite[p.\ 97]{FK94}), we obtain Table \ref{table:cayley-type-class-simplepart}.
  The case \(\g = \fe_{7(-25)}\), i.e.\ \(V = \Herm(3,\bbO)\), was considered in Example \ref{ex:jordan-simple-subalg-hermoctonions}, and the remaining cases follow analogously.
\end{proof}

\subsection{Reductions of the general case}
\label{sec:reductions}

In Proposition \ref{prop:cayley-class-reduction}, we have shown that the classification of the subalgebras \(\g(\tau,h)\) for Cayley type involutions \(\tau = e^{i\pi\ad H}\) with \(\spec(\ad H) = \{0, \pm 1\}\) and \(\g_1(h) \subset \g_1(H)\) can be reduced to the case where \(h \in \g^\tau\) is semisimple and induces a 5-grading on \(\g\). Our first goal in this section is to prove a more general version of this lemma which also holds for non-Cayley type involutions.

Recall that \(\g\) is a hermitian simple Lie algebra.

\begin{lem}
  \label{lem:tube-type-emb-lower-rank}
  Let \(h \in \g\) be a semisimple element with \(\spec(\ad h) \subset \{0, \pm \frac{1}{2}, \pm 1\}\). If \(\g_{\pm 1}(h) \neq \{0\}\), then
  \begin{equation}
    \label{eq:tube-type-emb-lower-rank}
    \g_t(h) := \g_{-1}(h) \oplus [\g_{-1}(h), \g_1(h)] \oplus \g_1(h)
  \end{equation}
  is a hermitian simple Lie algebra of tube type.
\end{lem}
\begin{proof}
  Fix a Cartan decomposition \(\g = \fk \oplus \fp\) of \(\g\) and a maximal abelian subspace \(\fa \subset \fp\) of dimension \(r = \rk_\R \g\).
  We may assume that \(h \in \fa\) because \(h\) is hyperbolic.
  The restricted root system of \(\g\) is either of type \((C_r)\) or of type \((BC_r)\) (cf.\ \eqref{eq:tube-type} and \eqref{eq:non-tube-type}).
  Let \(H_k\) be the coroot of \(2\varepsilon_k\) for \(k=1,\ldots,r\).
  Using the action of the Weyl group of \(\g\) on \(h\), we may assume that \(h = \sum_{k=1}^r \lambda_k H_k\) for \(\lambda_1 \geq \lambda_2 \geq \ldots \lambda_r \geq 0\).
  Then \(\lambda_k \in \{0,\frac{1}{2}\}\) for all \(1 \leq k \leq r\).
  Hence, Remark \ref{rem:tube-type-h-element}(c) implies that \(\g_t(h)\) is contained in an \(\ad(h)\)-invariant hermitian simple tube type Lie algebra \(\g_t\) with \(\rk_\R \g_t = r\).
  Thus, we may assume that \(\g\) is of tube type.
  We recall from Section \ref{sec:jordan-algebras} that, for the element \(H := \frac{1}{2}\sum_{k=1}^r H_k\), the eigenspace \(\g_1(H)\) can be endowed with the structure of a simple euclidean Jordan algebra.
  Now Lemma \ref{lem:caley-class-easy-case} implies that \(\g_1(h)\) is a simple euclidean Jordan algebra as well, so that the Lie algebra generated by \(\g_{\pm 1}(h)\) is hermitian simple and of tube type (cf.\ Example \ref{ex:KKT}).
\end{proof}

\begin{thm}
  \label{thm:invclass-5-grad-reduction}
  For a hyperbolic element \(h \in \g\), define
  \[\fs(h) := \fs_-(h) \oplus [\fs_-(h),\fs_+(h)] \oplus \fs_+(h) \quad \text{with} \quad \fs_\pm(h) := W_{\pm 1}(h) - W_{\pm 1}(h)\]
  and
  \[W_{\pm 1}(h) := (\pm \Wmin) \cap \g_{\pm 1}(h).\]
  If \(\fs(h) \neq \{0\}\), then it is a hermitian simple Lie algebra of tube type, and there exists \(h_0 \in \fs(h)\) such that
  \begin{itemize}
    \item \(h - h_0 \in \fz_\g(\fs(h))\),
    \item \(\fs(h) = \fs(h_0) = \g_{-1}(h_0) \oplus [\g_{-1}(h_0), \g_1(h_0)] \oplus \g_1(h_0)\), and
    \item \(h_0\) is hyperbolic with \(\spec(\ad h_0) \subset \{0, \pm \frac{1}{2},\pm 1\}\).
  \end{itemize}
\end{thm}
\begin{proof}
  {\bf Step 1:} We fix a Cartan involution \(\theta\) on \(\g\) such that \(\theta(h) = -h\).
  Then \(\fs(h)\) is \(\theta\)-invariant because \(\theta(\Wmin) = \Wmin\) and \(\theta \g_{\pm 1}(h) = \g_{\mp 1}(h)\).
  Let \(\hat \fs := \fs(h) + \R h\).
  Then \([h,\fs(h)] \subset \fs(h)\) implies that \(\hat \fs\) is a \(\theta\)-invariant subalgebra.
  Hence both subalgebras \(\fs(h)\) and \(\hat \fs\) are reductive. Moreover, \(\fs(h)\) is semisimple because \(\fs_\pm(h) \subset [h,\fs(h)] \subset [\hat \fs, \hat \fs] = [\fs(h),\fs(h)]\).
  In particular, we can decompose \(h\) into \(h = h_z + h_0\) with \(h_0 \in \fs(h) \cap \fp\) and \(h_z \in \fz(\hat \fs(h)) \cap \fp\).
  If \(h_0 = 0\), then \(\fs(h) = \{0\}\), which proves the claim. Hence, we assume from now on that \(h_0 \neq 0\), so that \(\ad h_0\) induces the same 3-grading on \(\fs(h)\) as \(h\).

  {\bf Step 2:} We show that \(h_0\) induces a 5-grading on \(\g\).
  The closed convex cone \(W_\fs := \Wmin \cap \fs(h)\) is obviously pointed and invariant.
  It is also generating because \(W_\fs - W_\fs \subset \fs(h)\) is an ideal in \(\fs(h)\) containing \(\fs_\pm(h)\). 
  In particular, \(\fs(h)\) is admissible, so that it can be written as a direct sum \(\fs_0 \oplus \bigoplus_{j=1}^m \fs_j\) with \(\fs_0\) compact and \(\fs_j\) hermitian simple for \(1 \leq j \leq m\) (cf.\ Lemma \ref{lem:adm-red-liealg-ideals}).
  Since \(\fs(h)\) is generated by the subspaces \(\fs_\pm(h)\), which consist of nilpotent elements, we have \(\fs_0 = \{0\}\).
  The element \(h_0\) decomposes into a sum \(h_0 = \sum_{j=1}^m h_j\) with \(h_j \in \fs_j \cap \fp\), and each \(h_j\) induces a 3-grading on \(\fs_j\).
  In particular, the ideals \(\fs_j\) are of tube type by Lemma \ref{lem:tube-type-3grad}.

  Moreover, the Cartan involution \(\theta\) leaves each \(\fs_j\) invariant, because otherwise there would be a pair of simple ideals \(\fs_j\) and \(\fs_k\) for \(1 \leq j < k \leq m\) such that \(\theta(\fs_j) = \fs_k\), so that \((\fs_j \oplus \fs_k)^\theta \cong \fs_j\) would be non-compact, which is a contradiction. Thus, \(\theta\) restricts to a Cartan involution on each simple ideal of \(\fs(h)\).

  Let \(U_j \in \Wmin \cap \fs_j\) be an \(H\)-element in \(\fs_j \cap \fk\) and let \(\kappa_j : (\fsl(2,\R), \frac{1}{2}U) \rightarrow (\fs_j, U_j)\) be an \((H_2)\)-homomorphism with \(\kappa_j(\frac{1}{2}H) = h_j\) (cf.\ Remark \ref{rem:tube-type-h-element}).
  Since \(U_\fs := \sum_{j=1}^m U_j\) is an \(H\)-element of \(\fs(h)\) by \cite[Rem.\ II.2]{HNO94}, the inclusions \(\iota_j : (\fs_j, U_j) \rightarrow (\fs(h), \sum_{k=1}^m U_k)\) are \((H_1)\)-homomorphisms, so that we obtain \((H_1)\)-homomorphisms \(\iota_j \circ \kappa_j : (\fsl(2,\R), \tfrac{1}{2}U) \rightarrow (\fs(h), U_\fs)\).
  The images of these homomorphisms commute, so that the commutative sum
  \[\kappa_\fs : (\fsl(2,\R), \tfrac{1}{2}U) \rightarrow (\fs(h), U_\fs), \quad x \mapsto \sum_{j=1}^m (\iota_j \circ \kappa_j) (x),\]
  is a \((H_1)\)-homomorphism that satisfies \(\kappa_\fs(\frac{1}{2}H) = \sum_{j=1}^m h_j = h_0\).
  Finally, \cite[Prop.\ II.9]{HNO94} implies that there exists an \(H\)-element \(U_\g \in \g\) such that the inclusion \(\iota_\fs: (\fs(h), U_\fs) \rightarrow (\g, U_\g)\) is an \((H_1)\)-homomorphism.
  By applying the arguments from \cite[p.\ 202]{HNO94} to the \((H_1)\)-homomorphism \(\kappa := \iota_\fs \circ \kappa_\fs\), which satisfies \(\kappa(\frac{1}{2}H) = h_0\), we see that \(\ad h_0\) induces a 5-grading
  \[\g = \g_{-1}(h_0) \oplus \g_{-\frac{1}{2}}(h_0) \oplus \g_0(h_0) \oplus \g_{\frac{1}{2}}(h_0) \oplus \g_1(h_0).\]
  This also shows that \(\ad h_0\) is semisimple with \(\spec(\ad h_0) \subset \{0, \pm \frac{1}{2}, \pm 1\}\).

  {\bf Step 3:} Since \(\g\) is hermitian simple and \(h_0\) is a semisimple element inducing a 5-grading on \(\g\), Lemma \ref{lem:tube-type-emb-lower-rank} shows that \(\g_t(h_0)\) is a hermitian simple Lie algebra of tube type.
  In particular, \(\g_1(h_0)\) can be endowed with the structure of a simple euclidean Jordan algebra and the convex cone \(\g_1(h_0) \cap \Wmin\) is generating in \(\g_1(h_0)\) by Lemma \ref{lem:jordan-herm-cone}.
  Thus, we have \(\fs(h_0) = \g_t(h_0)\). In particular, \(\fs(h_0)\) is hermitian simple.

  {\bf Step 4:} We note that \(\fs_\pm(h) \subset \g_{\pm 1}(h_0)\) because \([h_z, \fs_{\pm}(h)] = \{0\}\). It remains to show \(\g_{\pm 1}(h_0) \subset \fs_\pm(h)\), so that \(\fs(h) = \fs(h_0)\).
  To this end, we choose a maximal abelian subspace \(\fa_\fs \subset \fs(h) \cap \fp\) containing \(h_0\).
  Then there exists a maximal abelian subspace \(\fa \subset \fp\) of dimension \(r := \rk_\R \g\) containing \(\fa_\fs\) and \(h_z\) because \([h_z, \fa_\fs] \subset [h_z,\fs(h)] = \{0\}\).
  
  Recall from Remark \ref{rem:tube-type-h-element}(b) that we can extend \(\fa\) to a subalgebra \(\fh \cong \fsl(2,\R)^r\) which is the image of an \((H_1)\)-inclusion \(\fsl(2,\R)^r \rightarrow \g\).
  We identify \(\fh\) with \(\fsl(2,\R)^r\) and define \(\{\varepsilon_1,\ldots,\varepsilon_r\}\) as the dual basis of \(\{H_1,\ldots,H_r\} \subset \fa\), so that the restricted root system of \(\g\) is given by \eqref{eq:tube-type} or \eqref{eq:non-tube-type}.
  The image of \(\kappa_\fs\), respectively \(\kappa\), is contained in \(\fh\).
  Since \(h_0 = \kappa(\frac{1}{2}H) \in \fa\) induces a 5-grading on \(\g\), we may, after reordering or changing the sign of \(H_1,\ldots,H_r\) if necessary, assume that \(h_0 = \frac{1}{2}\sum_{k=1}^s H_k\) for some \(1 \leq s \leq r\) (cf.\ Remark \ref{rem:tube-type-3grad-sign}). We may also assume that \(h_z = \sum_{k=1}^r \mu_k H_k\) for some \(\mu_1,\ldots,\mu_r \in \R\) because \(h_z \in \fa\).
In particular,
\[\g_{\pm 1}(h_0) = \bigoplus_{1 \leq i < j \leq s} \g^{\pm (\varepsilon_i + \varepsilon_j)} \oplus \bigoplus_{1 \leq i \leq s} \g^{\pm 2 \varepsilon_i}.\]
Let \(x := \kappa(X), y := \kappa(Y) \in \fs(h) \cap \fh \subset \g\).
Then \(x = \sum_{k=1}^s X_k\) and \(y = \sum_{k=1}^s Y_k\) such that \((H_k, X_k, Y_k)\) is an \(\fsl(2)\)-triple with
\[[H_k,X_\ell] = 2\delta_{k \ell} X_\ell \quad \text{and} \quad [H_k, Y_\ell] = -2\delta_{k \ell} Y_\ell \quad \text{ for } 1 \leq k, \ell \leq s.\]
In particular, we have \(X_k \in \g^{2\varepsilon_k}\) and \(Y_k \in \g^{-2\varepsilon_k}\) for \(1 \leq k \leq s\).
Since \(x \in \fs(h)\), we have \([h_z, x] = 0\), so that
\[0 = [h_z, x] = \sum_{k=1}^s 2\mu_k X_k \]
implies that \(\mu_k = 0\) for \(1 \leq k \leq s\) and thus \([h_z, \g_1(h_0)] = \{0\}\).
Similarly, we see that \([h_z, \g_{-1}(h_0)] = \{0\}\) and thus \(\g_{\pm 1}(h_0) = \fs_\pm(h)\).
Since these subspaces generate \(\fs(h_0)\), respectively \(\fs(h)\), this proves that \(\fs(h) = \fs(h_0)\).
In particular, \(\fs(h)\) is a hermitian simple Lie algebra of tube type.
\end{proof}

\begin{cor}
  \label{cor:invclass-5grad-red}
  Let \(W \subset \g\) be a pointed generating invariant closed convex cone and let \(\tau \in \Aut(\g)\) be an involution with \(\tau(W) = -W\). Then, for every \(h \in \g^\tau\), there exists a hyperbolic element \(h_0 \in \g^\tau\) with \(\spec(\ad h_0) \subset \{0, \pm \frac{1}{2}, \pm 1\}\) and \(\g(\tau,h) = \g(\tau,h_0)\).
\end{cor}
\begin{proof}
  We may assume that \(h\) is hyperbolic by using the Iwasawa decomposition of \(\g\) (cf.\ Remark \ref{rem:hyperbolic-elements}).
  Let \(\fs := \fs(h)\) be defined as in Theorem \ref{thm:invclass-5-grad-reduction}.
  Then there exists a hyperbolic element \(h_0 \in \fs\) with \(\spec(\ad h_0) \subset \{0, \pm \frac{1}{2}, \pm 1\}\) and \(\fs = \g_{-1}(h_0) \oplus [\g_{-1}(h_0), \g_1(h_0)] \oplus \g_1(h_0)\).
  It remains to show that \(h_0 \in \g^\tau\).
  To this end, we first note that \(\fs\) is \(\tau\)-invariant because \(\tau(W) = -W\) and \(\ad h\) commutes with \(\tau\).
  Since \(h - h_0 \in \fz_\fg(\fs)\), the adjoint representations of \(h\) and \(h_0\) coincide on \(\fs\), so that \(\ad(h_0) \circ \tau = \tau \circ \ad(h_0)\) on \(\fs\).
  This implies \(\ad_\fs(h_0) = \ad_\fs(\tau(h_0))\), so that \(\tau(h_0) = h_0\) because \(\fs\) is simple by Theorem \ref{thm:invclass-5-grad-reduction}.
\end{proof}

\begin{lem}
  \label{lem:invclass-nontubetype-red}
  Let \(\tau \in \Aut(\g)\) be an involution such that \(\tau(\Wmin) = -\Wmin\).
  Let \(h \in \g^\tau\) be a hyperbolic element with \(\spec(\ad h) = \{0, \pm \frac{1}{2}, \pm 1\}\).
  Then the subalgebra
  \[\g_t(h) := \g_{-1}(h) \oplus [\g_{-1}(h), \g_1(h)] \oplus \g_1(h)\]
  is invariant under \(\tau\) and hermitian simple and of tube type with \(h \in \g_t(h)\). Moreover, we have \(\g(\tau,h) = \g_t(h)(\tau\lvert_{\g_t(h)},h)\).
\end{lem}
\begin{proof}
  The \(\tau\)-invariance of \(\g_t(h)\) is a consequence of \(\tau \circ \ad h = \ad h \circ \tau\).
  Lemma \ref{lem:tube-type-emb-lower-rank} implies that \(\g_t(h)\) is a hermitian simple Lie algebra of tube type.
  Since \(h\) is hyperbolic and induces a 5-grading on \(\g\), it is contained in the range of an \((H_1)\)-inclusion \(\fsl(2,\R)^r \rightarrow \g\), where \(r = \rk_\R \g\), and there exist nilpotent elements \(x \in \g_1(h)\) and \(y \in \g_{-1}(h)\) of convex type such that \((h,x,y)\) is an \(\fsl(2)\)-triple and \(x,y \in \Wmin(\g)\) (cf.\ \cite[Thm.\ III.9]{HNO94}). 
  Hence, \(\Wmin(\g) \cap \g_t(h)\) is non-zero and therefore generating in \(\g_t(h)\), so that \(\Wmin(\g) \cap \g_{\pm 1}(h) = \Wmin(\g_t(h)) \cap (\g_t(h))_{\pm 1}(h)\) because all nilpotent elements of convex type in \(\g_t(h)\) are contained in \(\Wmin(\g_t(h)) \cup (-\Wmin(\g_t(h)))\).
  This proves that \(\g(\tau,h) = \g_t(h)(\tau\lvert_{\g_t(h)},h)\).
\end{proof}

In the case of 3-gradings, the following lemma translates our classification problem into the Jordan algebra context:

\begin{prop}
  \label{prop:invclass-jordan-inv}
  Suppose that \(\g\) is of tube type, let \(\theta\) be a Cartan involution of \(\g\), and let \(0 \neq h \in \fp\) be such that \(\g = \g_{-1}(h) \oplus \g_0(h) \oplus \g_1(h)\).
  \begin{enumerate}
    \item Let \(\tau \in \Aut(\g)\) be an involution with \(\tau(\Wmin) = -\Wmin\), \(\theta \circ \tau = \tau \circ \theta\), and \(\tau(h) = h\).
      Moreover, endow \(V := \g_1(h)\) with the structure of a euclidean simple Jordan algebra as in {\rm Example \ref{ex:KKT}}.
      Then \(-\tau\lvert_V \in \Aut(V)\).
    \item Conversely, for every Jordan algebra involution \(\sigma \in \Aut(V)\), there exists a unique extension to an involutive automorphism \(\sigma_\g \in \Aut(\g)\) such that \(\theta \circ \sigma_\g = \sigma_\g \circ \theta\) and \(\sigma_\g(h) = h\).
  \end{enumerate}
\end{prop}
\begin{proof}
  (a) We may assume that the Jordan algebra structure on \(V\) is chosen in such a way that \(\Wmin \cap V = \oline{\Omega_V}\) (cf.\ Lemma \ref{lem:jordan-herm-cone}).
  Recall from Example \ref{ex:KKT} that \(V\) is a simple euclidean Jordan algebra and that the associative scalar product is given by \(\la x, y\ra := -\beta(x,\theta y)\, x,y \in V\), where \(\beta\) is the Cartan--Killing form of \(\g\).
  Combining these facts, we see that \(-\tau\) is an automorphism of the irreducible symmetric cone \(\Omega_V\) and is orthogonal.
  Now the claim follows from \(G(\Omega_V) \cap O(V) \subset \Aut(V)\) (cf.\ \cite[p.\ 55]{FK94}).

  (b) Let \(\sigma \in \Aut(V)\) be an involutive automorphism of \(V\).
  Then, since \(\g\) is isomorphic to the Lie algebra obtained from \(V\) with the Kantor--Koecher--Tits construction (cf.\ Example \ref{ex:KKT}), there exists an extension \(\sigma_\g\) of \(\sigma\) to an involutive automorphism of \(\g\) which is uniquely determined by
  \[\sigma_\g(\theta x) := \theta \sigma(x) \quad \text{and} \quad \sigma_\g(x) := \sigma(x) \quad \text{for } x \in V\]
  (cf. \cite[p.\ 797f.]{Koe67}). In particular, \(\sigma_\g\) preserves the grading of \(\g\) and therefore the eigenspaces of \(\ad h\), so that \(\sigma_\g\) commutes with \(\ad h\). Since \(\g\) is  simple, this implies \(\sigma_\g(h) = h\).
\end{proof}

\begin{rem}
  \label{rem:jordan-inv-ext-lie}
  Let \(\g\), \(\theta\), and \(h \in \fp\) be defined as in Proposition \ref{prop:invclass-jordan-inv}.

  Let \(\tau_h := e^{i\pi \ad h}\) be the Cayley type involution induced by \(h\) and let \(\tau \in \Aut(\g)\) be an involution with \(\tau(h) = h\) and \(\tau(\Wmin) = -\Wmin\) that commutes with \(\theta\). Note that \(\tau\) and \(\tau_h\) also commute.
  Then \(\tau_h\lvert_V = -\id_V\), so that \(-\tau\lvert_V = (\tau_h \tau)\lvert_V\).
  Conversely, the extension of \(-\tau\lvert_V\) to \(\g\) that we obtain from Proposition \ref{prop:invclass-jordan-inv}(b) equals \(\tau\tau_h\).
  Hence, every involution \(\tau \in \Aut(\g)\) for which there exists an element \(h \in \g^{\tau, -\theta}\) that induces a 3-grading on \(\g\)  is a product of the form \(\tau = \tau_h \circ \sigma_\g = \sigma_\g \circ \tau_h\), where \(\sigma_\g\) is an extension of an involutive Jordan algebra automorphism \(\sigma \in \Aut(V)\).
\end{rem}

Let \(V\) be a simple euclidean Jordan algebra and let \(\sigma \in \Aut(V)\) be an involutive automorphism.
Then \(V^\sigma\) is a euclidean subalgebra of \(V\).
We recall from Proposition \ref{prop:class-inv-jordan}(a) that \(V^\sigma\) is either simple or the direct sum of two simple euclidean Jordan algebras.

In the context of Proposition \ref{prop:invclass-jordan-inv}, with \(\sigma := -\tau\lvert_V\), we have \(V^\sigma = \g_1^{-\tau}(h)\). In particular, \(\Wmin \cap \g_1^{-\tau}(h)\) is up to a sign the set of squares in the Jordan subalgebra \(V^\sigma\) (cf.\ Lemma \ref{lem:jordan-herm-cone}), hence is generating in \(\g_1^{-\tau}(h)\).
By applying the Cartan involution \(\theta\) on \(\g_1^{-\tau}(h)\) and the Kantor--Koecher--Tits construction on the Jordan algebra \(V^\sigma\), we obtain the theorem below. Note that this result does not rely on the classification of involutive automorphisms of real euclidean simple Jordan algebras, but only the structural result in Proposition \ref{prop:class-inv-jordan}(a).

\begin{thm}
  \label{thm:invclass-3grad-jordan}
  Suppose that \(\g\) is of tube type, let \((\tau,\theta) \in \cA(\g)\) (cf. \eqref{eq:invclass-def-aux-set}) and let \(0 \neq h \in \g^{\tau,-\theta}\) be such that \(\g = \g_{-1}(h) \oplus \g_0(h) \oplus \g_1(h)\).
  Then
  \[\g(\tau,h) = \g_{-1}^{-\tau}(h) \oplus [\g_{-1}^{-\tau}(h), \g_1^{-\tau}(h)] \oplus \g_1^{-\tau}(h).\]
  Moreover, \(\g(\tau,h)\) is either hermitian simple and of tube type or the direct sum of two hermitian simple Lie algebras of tube type.
\end{thm}

\subsection{Involutions on hermitian simple Lie algebras of tube type and the KKT-construction}
\label{sec:inv-herm-kkt}

Proposition \ref{prop:invclass-jordan-inv} and Theorem \ref{thm:invclass-3grad-jordan} suggest that we take a closer look at the relationship between involutive automorphisms of simple euclidean Jordan algebras and their extensions to involutive automorphisms of the hermitian simple Lie algebras that we obtain from the KKT-construction (cf.\ Remark \ref{rem:tkk-onetoone}).

We recall that we assume \(\g\) to be a hermitian simple Lie algebra of real rank \(r\). In this section, we also fix a pair \((\tau,\theta) \in \cA(\g)\).

\begin{rem}
  The equivalence classes of involutive automorphisms \(\tau\) of \(\g\) can be characterized in terms of their fixed point algebras \(\g^\tau\) (cf.\ \cite{B57}).
  But two involutive automorphisms \(\tau_1,\tau_2\) on \(\g\) with isomorphic fixed point algebras can lead to non-isomorphic subalgebras \(\g(\tau_1,h)\) and \(\g(\tau_2,h)\) for \(\tau_1(h) = h = \tau_2(h)\):

  Suppose that \(\g\) is of tube type and of real rank \(r > 1\). Choose a Cartan decomposition \(\g = \fk \oplus \fp\) of \(\g\), a maximal abelian subspace \(\fa \subset \fp\), and coroots \(\{H_1,\ldots,H_r\}\) of \(\{2\varepsilon_1,\ldots,2\varepsilon_r\}\) (cf.\ \eqref{eq:tube-type}).
  Let \(H := \frac{1}{2}\sum_{k=1}^r H_k\) and set \(\tau_H := e^{i\pi\ad H}\). Then \(\tau_H\) is a Cayley type involution with \(\tau_H(\Wmin) = -\Wmin\), and we know from Theorem \ref{thm:cayley-type-class-simplepart} that \(\g(\tau_H,H) = \g\). Moreover, \(\ad H\) induces a 3-grading on \(\g\).

  As the element \(H_1 \in \fa\) is integral hyperbolic, we may also consider the involution \(\tau_{H_1} := e^{i\pi \ad H_1}\). Let \(\{X_1,\ldots,X_r\}\) be a Jordan frame in \(V := \g_1(H)\) as in Example \ref{ex:herm-jordan-frame}.
  By inspecting the action of \(\tau_{H_1}\) on the Peirce decomposition of \(V\) with respect to this Jordan frame (cf.\ Remark \ref{rem:herm-frame-roots}), we see that \(V^{\tau_{H_1}} = V^{(1)} \oplus V^{(r-1)}\).
  For each \(k \in \{1,\ldots,r\}\), we have \(\tau_{H_1}(X_k) = X_k\), so that \(\tau_{H_1}(\Wmin) = \Wmin\).
  As a result, Theorem \ref{thm:invclass-3grad-jordan} shows that \(\g(\tau_H \circ \tau_{H_1}, H)\) is the sum of two hermitian simple Lie algebras of tube type.

The involution \(\tau_H \circ \tau_{H_1}\) is induced by the integral hyperbolic element \(\frac{3}{2}H_1 + \frac{1}{2}\sum_{k=2}^r H_k\) and flips the cone \(\Wmin\), so that \(\g^{\tau_H} \cong \g^{\tau_H \circ \tau_{H_1}}\) by Proposition \ref{prop:inv-cayley-type}.
\end{rem}

\begin{lem}
  \label{lem:peirce-refl-cayley-type}
  Suppose that \(\g\) is of tube type and let \(0 \neq h \in \g^{\tau,-\theta}\) be such that \(\g = \g_{-1}(h) \oplus \g_0(h) \oplus \g_1(h)\). Endow \(V := \g_1(h)\) with the structure of a simple euclidean Jordan algebra as in {\rm Example \ref{ex:KKT}} and define \(\sigma := -\tau\lvert_V \in \Aut(V)\).
  If \(\sigma\) is a Peirce reflection {\rm (Definition \ref{def:inv-jordan})}, then \(\tau\) is of Cayley type.
\end{lem}
\begin{proof}
  Suppose that \(\sigma\) is a Peirce reflection. Then there exists a Jordan frame \(F := \{c_1,\ldots,c_r\}\) of \(V\) such that \(\sigma(c_k) = c_k\) for all \(1 \leq k \leq r\) and \(\sigma = P(w)\), where \(w = c_1 + \ldots + c_s\) for some \(1 \leq s \leq r\). In particular, for \(x \in V_{ij}\), we have \(\sigma(x) = x\) if \(1 \leq i < j \leq s\) or \(s < i < j \leq r\) and \(\sigma(x) = -x\) if \(1 \leq i \leq s < j \leq r\) (cf.\ \cite[1.7.1]{BH98}).

  Identify \(\g_0(h)\) with the structure algebra \(\str(V) = \Der(V) \oplus L(V)\) of \(V\) (cf.\ Remark \ref{rem:tkk-onetoone}).
  Then \(\fa := \spann(L(F))\) is maximal abelian in \(\fp\), the root system \(\Sigma \subset \fa^*\) is of type \((C_r)\) as in \eqref{eq:tube-type}, and the elements \(2L(c_k)\) are coroots of \(2\varepsilon_k\) for \(1 \leq k \leq r\) (cf.\ Remark \ref{rem:herm-frame-roots}).
  The space \(\fa\) is contained in \(\g^\tau\): To see this, recall from Remark \ref{rem:jordan-inv-ext-lie} that the unique extension of \(-\tau\lvert_V\) with the properties from Proposition \ref{prop:invclass-jordan-inv} is given by \(\tau\circ\tau_h = \tau_h \circ \tau\), where \(\tau_h := e^{i\pi\ad h}\). Thus, we have for \(x \in \spann(F)\):
  \[(\tau_h \circ \tau)L(x) = L((\tau_h \circ \tau)x) = L(-\tau(x)) = L(\sigma(x)) = L(x)\]
  (cf.\ \cite[p.\ 795]{Koe67}). But since \(L(x) \in \str(V) \subset \g_0(h)\), we have \(\tau_h(L(x)) = L(x)\), which shows \(L(x) \in \g^\tau\).

  Let \(h' := \sum_{k=1}^s L(c_k)\) and set \(\tau_{h'} := e^{i\pi\ad h'}\). Then we have \(\ad h \circ \tau_{h'} = \tau_{h'} \circ \ad h\), and \(\tau_{h'}\) and \(\sigma\) coincide on \(V\), so that \(\tau_{h'} \circ \tau_h = \tau\) (cf.\ Remark \ref{rem:jordan-inv-ext-lie}). Since \(\tau_{h'} \circ \tau_h\) is a Cayley type involution, this proves the claim.
\end{proof}

To prepare the following theorem, we first recall that involutive automorphisms \(\tau \in \Aut(\g)\) with \(\tau(\Wmin) = -\Wmin\) for a hermitian simple Lie algebra \(\g \not\cong \so(2,n)\) of tube type can be divided into three (disjoint) classes: Cayley type involutions, involutions \(\tau\) with the property that \(\rk_\R \g^\tau = \rk_\R \g\) and \(\g^\tau\) is simple, and those with \(\rk_\R \g = 2\rk_\R \g^\tau\) (cf.\ \cite[Thm.\ 3.2.8]{HO97}).

\begin{thm}
  \label{thm:inv-hermlie-jordan}
  Suppose that \(\g\) is of tube type and let \(0 \neq h \in \g^{\tau,-\theta}\) be such that \(\g = \g_{-1}(h) \oplus \g_0(h) \oplus \g_1(h)\). Endow \(V := \g_1(h)\) with the structure of a simple euclidean Jordan algebra as in {\rm Example \ref{ex:KKT}} and define \(\sigma := -\tau\lvert_V \in \Aut(V)\). Then the following assertions hold:
  \begin{enumerate}
    \item If \(\tau\) is a Cayley type involution, then \(\sigma\) is a Peirce reflection.
    \item If \(\tau\) is of non-Cayley type with \(\rk_\R \g^\tau = \rk_\R \g\), then \(V^\sigma\) is simple and non-trivial with \(\rk V^\sigma = \rk V\).
    \item If \(\rk_\R \g = 2\rk_\R \g^\tau\), then \(\sigma\) is non-split {\rm (Definition \ref{def:inv-jordan})}.
  \end{enumerate}
\end{thm}
\begin{proof}
  (a) Suppose that \(\tau\) is a Cayley type involution. Then, according to \cite[Lem.\ 1.3.10, Prop.\ 3.1.14]{HO97}, there exists an element \(h_0 \in \fz(\g^\tau) \cap \fp\) such that \(\tau = e^{i\pi\ad h_0}\) and \(\g = \g_{-1}(h_0) \oplus \g_0(h_0) \oplus \g_1(h_0)\). In particular, we have \([h,h_0] = 0\).
  Choose a maximal abelian subspace \(\fa \subset \fp\) of dimension \(r = \rk_\R \g\) containing \(h\) and \(h_0\). The restricted root system \(\Sigma \subset \fa^*\) of \(\g\) is then of type \((C_r)\) as in \eqref{eq:tube-type}. Let \(H_1,\ldots,H_r \in \fa\) be the coroots of \(2\varepsilon_1,\ldots,2\varepsilon_r \in \Sigma\) respectively. Since \(h\) and \(h_0\) induce 3-gradings on \(\g\), we may assume that the coroots are chosen in such a way that
  \[h = \frac{1}{2}\sum_{k=1}^r H_k \quad \text{and} \quad h_0 = \frac{1}{2}\left(\sum_{k=1}^\ell H_k - \sum_{k=\ell + 1}^r H_k\right)\]
  for some \(0 \leq \ell \leq r\) (cf.\ Remark \ref{rem:tube-type-3grad-sign}). Moreover, we may choose a Jordan frame \(F = \{X_1,\ldots,X_r\}\) as in Remark \ref{rem:herm-frame-roots} such that \eqref{eq:herm-frame-roots} holds. Now an inspection of the action of \(\sigma\) on the root spaces shows that, on \(V_{ij}\), the involution \(\sigma\) acts by the identity map if \(1 \leq i,j \leq \ell\) or \(\ell < i,j \leq r\), and by multiplication with \((-1)\) otherwise. Hence, \(-\tau\) is a Peirce reflection on \(V\) (cf.\ \cite[1.7.1]{BH98}).

  (b) We first show that \(\sigma\) must be a split involution: Since \(\rk_\R \g^\tau = \rk_\R \g\), we can choose a maximal abelian subspace \(\fa \subset \fp\) such that \(\fa \subset \g^\tau\) and \(h \in \fa\).
  Choose \(H_1,\ldots,H_r \in \fa\) and the Jordan frame \(F = \{X_1,\ldots,X_r\}\) as in (a).
  Then \(\tau(H_k) = H_k\) for all \(1 \leq k \leq r\) implies that \(\tau\) preserves all root spaces.
  Since \(\g^{2\varepsilon_k} = \R X_k\) and \(X_k \in \Wmin \cup (-\Wmin)\), we must have \(\tau(X_k) = -X_k\), i.e.\ \(\sigma(X_k) = X_k\).
  In particular, \(\sigma\) preserves the subalgebra \(V^{(1)} = \R X_k\), hence it can not be non-split by Lemma \ref{lem:jordan-nonsplit-inv-subalg}.
  Thus, we must have \(\rk V^\sigma = \rk V\).

  We already observed in Lemma \ref{lem:peirce-refl-cayley-type} that \(\sigma\) cannot be a Peirce reflection, so that \(V^\sigma\) must be simple by Proposition \ref{prop:class-inv-jordan}.

  (c) Suppose that \(\sigma\) is a split involution, i.e.\ \(\rk V = \rk V^\sigma\). By Proposition \ref{prop:class-inv-jordan}, there exists a Jordan frame \(F \subset V\) such that \(\sigma(F) = F\) and, for \(R := \spann(F)\), we have \(R = R^\sigma\).
  Recall from Remark \ref{rem:tkk-onetoone} that we can identify \(\g_0(h)\) with the structure algebra \(\str(V) = \Der(V) \oplus L(V)\), and that this decomposition is a Cartan decomposition of \(\str(V)\), so that we have \(L(V) \subset \fp\).
  Now our assumptions imply that \(L(R)\) is an abelian subspace of dimension \(r = \rk_\R \g\) contained in \(\g^\tau \cap \fp\), which is a contradiction because \(\rk_\R \g^\tau = \frac{1}{2}\rk_\R \g\).
\end{proof}

Let \(\g, \tau, h,\) and \(V\) be as above and suppose that \(\tau\) is not of Cayley type.
If \(\rk V \neq 2\), then there exists at most one equivalence class of involutions \(\sigma \in \Aut(V)\) such that \(\sigma\) is non-split and at most one equivalence class such that \(V^\sigma\) is simple with \(\rk V = \rk V^\sigma\) (cf.\ \cite[Table 1.5.1]{BH98}). 
Thus, Theorem \ref{thm:inv-hermlie-jordan} allows us to determine the fixed point algebra of the involution \(-\tau\lvert_V\) depending on \(\g^\tau\).

In the case where \(\rk V = 2\), i.e.\ \(V \cong \R \rtimes \R^{n}\) for some \(n \in \N\), we may have  several equivalence classes of split involutions which are not Peirce reflections, as we have seen in Example \ref{ex:inv-jordan-minkowski}, so that we need a more elaborate argument here:

\begin{prop}
  \label{prop:jordan-inv-herm-rk2}
  Let \(\g = \so(2,n)\) for \(n \in \N, n \neq 2,\) and let \((\tau,\theta) \in \cA(\g)\). Let \(0 \neq h \in \g^{\tau,-\theta}\) be such that \(\g = \g_{-1}(h) \oplus \g_0(h) \oplus \g_1(h)\). Endow \(V := \g_1(h)\) with the structure of a simple euclidean Jordan algebra as in {\rm Example \ref{ex:KKT}}. If \(\g^\tau \cong \so(1,k) \oplus \so(1,n-k)\) for \(1 < k < n-1\), then either \(V^{-\tau} \cong \R \times \R^{n-k}\), or \(V^{-\tau} \cong \R \times \R^k\).
\end{prop}
\begin{proof}
  From the classification of causal symmetric pairs \cite[Thm.\ 3.2.8]{HO97} we deduce that \(\tau\) is not a Cayley type involution, so that \(V^{-\tau}\) must be simple with \(\rk V^{-\tau} = 2\) by Theorem \ref{thm:inv-hermlie-jordan}.
  
  The element \(h\) induces a 3-grading on \(\g^\tau \cong \so(1,k) \oplus \so(1,n-k)\). Suppose that \(h\) induces a 3-grading on both simple ideals of \(\g^\tau\).
  According to \cite[p.\ 893f.]{KN64}, we would then have \(\g^\tau_1(h) = V^\tau \cong \R^{k-1} \oplus \R^{n-k-1} = \R^{n-2}\), so that \(V^{-\tau} \cong \R \times \R\) is a two-dimensional euclidean simple Jordan algebra of rank \(2\). But this leads to a contradiction because such a Jordan algebra does not exist.

  Hence, \(h\) commutes with one of the simple ideals of \(\g^\tau\), so that it is contained in one of them. We identify \(\g^\tau\) with the above direct sum.
  If \(h \in \so(1,k)\), then \(\g^\tau_1(h) = V^\tau \cong \R^{k-1}\), so that \(V^{-\tau} \cong \R \times \R^{(n-k+1)-1}\) (cf.\ \cite[Table 1.5.1]{BH98}).

  On the other hand, if \(h \in \so(1,n-k)\), then we have \(\g^\tau_1(h) = V^\tau \cong \R^{n-k-1}\) and therefore \(V^{-\tau} \cong \R \times \R^{(k+1)-1}\) (cf.\ \cite[p.\ 893f.]{KN64}).
\end{proof}

Recall from Section \ref{sec:inv-jordan} that equivalence classes of involutions on simple euclidean Jordan algebras can be divided into split and non-split involutions, and split involutions further divide into Peirce reflections and non-trivial involutions with a simple fixed point algebra. A crucial step in the proof of Theorem \ref{thm:herm-invclass} is to determine for each involution \(\tau \in \Aut(\g)\) of a hermitian simple Lie algebra \(\g\) of tube type the type of the Jordan algebra involution \(-\tau\) on the corresponding simple Jordan algebra. That such a Jordan algebra always exists is a consequence of the following proposition:

\begin{prop}
  \label{prop:inv-herm-basis}
  There exists a \(\tau\)-invariant maximal abelian subspace \(\fa \subset \fp\) with the following properties:
  \begin{enumerate}
    \item Either \(\fa_\fh := \fa \cap \g^\tau = \fa\) or \(r = 2s\) is even and there exists a basis \(\{H_1,\ldots,H_r\}\) of \(\fa\) with the following properties:
      \begin{itemize}
        \item \(\tau(H_\ell) = H_{\ell+s}\) for \(1 \leq \ell \leq s\),
        \item the vectors \(K_\ell := H_\ell + H_{\ell + s}\) form a basis of \(\fa \cap \g^\tau\), and
        \item if we denote by \(\{\varepsilon_1,\ldots,\varepsilon_r\} \subset \fa^*\) the dual basis of \(\{H_1,\ldots,H_r\}\), then the restricted root system \(\Sigma \subset \fa^*\) is of type \((C_r)\) as in \eqref{eq:tube-type} if \(\g\) is of tube type and of type \((BC_r)\) as in \eqref{eq:non-tube-type} if \(\g\) is of non-tube type.
      \end{itemize}
    \item Let \(h_0 \in \fa_\fh\) be such that \(\spec(\ad h_0) \subset \{0, \pm\frac{1}{2},\pm 1\}\) and
      \[\g_t(h_0) := \g_{-1}(h_0) \oplus [\g_{-1}(h_0),\g_1(h_0)] \oplus \g_1(h_0)\]
      satisfies \(\rk_\R \g_t(h_0) = \rk_\R \g\). Then \(\fa_\fh \subset [\g_{-1}^{-\tau}(h_0),\g_1^{-\tau}(h_0)]\).
  \end{enumerate}
\end{prop}
\begin{proof}
  The existence of a maximal abelian subspace \(\fa \subset \fp\) with the properties in (a) follows from the construction in \cite[Sec.\ 4]{O91} and in particular \cite[Lem.\ 4.3]{O91} as well as the classification of irreducible symmetric pairs of hermitian type (cf.\ e.g.\ \cite[p. 208]{O91}). The arguments in \cite[Sec.\ 4]{O91} also show that there exists \(X_{\tilde \ell}^\pm \in \g^{\pm 2\varepsilon_{\tilde \ell}}\) such that \(H_{\tilde \ell} = [X_{\tilde \ell}^+, X_{\tilde \ell}^-]\) for \(1 \leq \tilde \ell \leq r\) and we have
  \begin{equation}
    \label{eq:prop-inv-herm-basis}
    \tau(X_\ell^\pm) = -X_\ell^\pm ~\text{ if } ~ \tau(H_\ell) = H_\ell \quad \text{and} \quad \tau(X_\ell^\pm) = -X_{\ell + s}^\pm ~\text{ if } ~ \tau(H_\ell) = H_{\ell + s}
  \end{equation}
  for \(1 \leq \ell \leq r\) if \(\fa_\fh = \fa\), respectively \(1 \leq \ell \leq s\) if \(\dim \fa_\fh = s = \frac{r}{2}\).
  
  Suppose that \(h_0 \in \fa_\fh\) has the properties stated in (b).
  If \(\fa_\fh = \fa\), then \(h_0\) is of the form \(h_0 = \frac{1}{2}\sum_{\ell=1}^r \alpha_\ell H_\ell\) with \(\alpha_\ell \in \{\pm 1\}\), and we have \(\tau(H_\ell) = H_\ell\) for \(1 \leq \ell \leq r\).
  Now \eqref{eq:prop-inv-herm-basis} shows that \(H_\ell \in [\g_{-1}^{-\tau}(h_0),\g_1^{-\tau}(h_0)]\)

  If \(r = 2s\) and \(\rk_\R \g^\tau = s\), then \(h_0\) must be of the form \(h_0 = \sum_{\ell=1}^s \alpha_\ell K_\ell\) with \(\alpha \in \{-1,1\}\).
  Equation \eqref{eq:prop-inv-herm-basis} then shows that
  \[\tau(X_\ell^\pm + X_{\ell + s}^\pm) = -(X_\ell^\pm + X_{\ell + s}^\pm)\]
  and therefore \(K_\ell \in [\g_{-1}^{-\tau}(h_0),\g_1^{-\tau}(h_0)]\) for \(1 \leq \ell \leq s\), which proves (b).
\end{proof}

\begin{cor}
  \label{cor:invclass-5grad-inclusion}
  There exists an element \(h \in \g^{\tau,-\theta}\) such that \(h\) is hyperbolic with \(\spec(\ad h) \subset \{0, \pm\frac{1}{2}, \pm 1\}\) and
  \[\g_t(h) = \g_{-1}(h) \oplus [\g_{-1}(h), \g_1(h)] \oplus \g_1(h)\]
  is a hermitian simple Lie algebra of tube type with \(\rk_\R \g = \rk_\R \g_t(h)\).
  If \(\g\) is of tube type, then \(\g = \g_t(h)\).
  Moreover, for every hyperbolic \(h_0 \in \g^{\tau,-\theta}\) with \(\spec(\ad h_0) = \{0, \pm\frac{1}{2}, \pm 1\}\), there exists an element \(h \in \g^{\tau,-\theta}\) with the above properties and
  \[\g_{\pm 1}(h_0) \subset \g_{\pm 1}(h),\quad [h,h_0] = 0,\quad \text{and}\quad h,h_0 \in [\g_{-1}^{-\tau}(h), \g_1^{-\tau}(h)] \subset \g_t(h).\]
\end{cor}
\begin{proof}
  Let \(\fa_\fh\) be the maximal abelian subspace of \(\g^\tau \cap \fp\) from Proposition \ref{prop:inv-herm-basis} and set \(s := \dim \fa_\fh\).
  If \(\fa_\fh\) is maximal abelian in \(\fp\), then the restricted root system \(\Sigma \subset \fa^*\) is either of type \((C_r)\) as in \eqref{eq:tube-type} or of type \((BC_r)\) as in \eqref{eq:non-tube-type}. Let \(K_\ell\) be the coroot of \(2\varepsilon_\ell\) for \(1 \leq \ell \leq s\).
  If \(\fa_\fh\) is not maximal abelian in \(\fp\), then we denote by \(K_\ell, 1 \leq \ell \leq s,\) the basis of \(\fa_\fh\) constructed in Proposition \ref{prop:inv-herm-basis}.
  Then the element \(h := \frac{1}{2}\sum_{\ell=1}^s K_\ell \in \g^\tau\) satisfies the above properties.

  Let now \(h_0 \in \g^{\tau,-\theta}\) be hyperbolic with \(\spec(\ad h_0) = \{0, \pm\frac{1}{2}, \pm 1\}\). By applying an inner automorphism in \(e^{\ad \g^{\tau, \theta}}\) to \(h_0\), we may assume that \(h_0 \in \fa_\fh\). The condition on the spectrum of \(\ad(h_0)\) then implies that, after renumbering \(K_1,\ldots,K_\ell\) if necessary, it must be of the form \(h_0 = \sum_{\ell=1}^s \lambda_\ell K_\ell\) with \(\lambda_\ell \in \{0, \pm \frac{1}{2}\}\). Hence, the element
  \[h := \sum_{\ell=1}^s \mu_\ell K_\ell \quad \text{with} \quad \mu_\ell := \begin{cases} \lambda_\ell & \text{if } \lambda_\ell \neq 0, \\ \frac{1}{2} & \text{if } \lambda_\ell = 0,\end{cases} \, (1 \leq \ell \leq s),\]
  has the above properties and \(\g_{\pm 1}(h_0) \subset \g_{\pm 1}(h)\).

  The remaining properties follow from \(h,h_0 \in \fa_\fh\) and Proposition \ref{prop:inv-herm-basis}(b).
\end{proof}

\begin{rem}
  \label{rem:jordan-inv-herm-rk2}
  Let \(\g = \so(2,n)\) and \((\tau,\theta) \in \cA(\g)\) be as in Proposition \ref{prop:jordan-inv-herm-rk2} and identify \(\g^\tau\) with \(\so(1,k) \oplus \so(1,n-k)\), where \(1 < k < n-1\).
  For each of the two proper simple ideals \(I\) in \(\g^\tau\), there exists an element \(h_0 \in I \cap \fp\) such that \(\ad h_0\) induces a 3-grading on \(I\) (cf. \cite{KN64}). In particular, \(h_0\) is a hyperbolic element in \(\g\).

  A combinatorial argument using the root system \eqref{eq:tube-type} of \(\so(2,n)\) shows that \(\spec(\ad_\g h_0) \subset \{0, \pm \frac{1}{2}, \pm 1\}\).
  Thus, by Corollary \ref{cor:invclass-5grad-inclusion}, there exists \(h \in \g^{\tau,-\theta}\) such that \(\g = \g_{-1}(h) \oplus \g_0(h) \oplus \g_1(h)\) and \(\g_1(h_0) \subset \g_1(h)\).
    In particular, the proof of Proposition \ref{prop:jordan-inv-herm-rk2} shows that \(h\) must be contained in \(I\) because \(\g_1(h) \cap I \neq \{0\}\).
    Hence, both cases \(V^{-\tau} \cong \R \times \R^{n-k}\) and \(V^{-\tau} \cong \R \times \R^k\) in Proposition \ref{prop:jordan-inv-herm-rk2} can occur.
\end{rem}

\subsection{Involutions on hermitian simple Lie algebras of non-tube type}

In this section, we assume that \(\g\) is of non-tube type and let \((\tau,\theta) \in \cA(\g)\).
According to Corollary \ref{cor:invclass-5grad-inclusion}, there exists a hyperbolic element \(h \in \g^{\tau,-\theta}\) such that \(\g_t(h) := \g_{-1}(h) \oplus [\g_{-1}(h), \g_1(h)] \oplus \g_1(h)\) is a hermitian simple subalgebra of tube type with \(\rk_\R \g = \rk_\R \g_t(h)\) and \(\spec(\ad h) = \{0, \pm\frac{1}{2},\pm 1\}\).
Since the involution \(\tau\) commutes with \(\ad h\), it restricts to an involutive automorphism on \(\g_t(h)\) with \(\tau(\Wmin(\g_t(h))) = -\Wmin(\g_t(h))\).
In this section, we want to determine \(\g_t(h)^\tau\) from \(\g^\tau\).
To this end, we first observe that
\[\g_t(h)^\tau = \g^\tau \cap \g_t(h) = \g^\tau_{-1}(h) \oplus [\g^\tau_{-1}(h),\g^\tau_1(h)] \oplus \g^\tau_1(h).\]
From \cite[Thm.\ 3.2.8]{HO97}, we see that \(\g^\tau \neq \g_t(h)^\tau\) and that \(\g^\tau\) is semisimple. Hence, the element \(h\) must induce a 5-grading on \(\g^\tau\).
The subalgebra \(\g_t(h)^\tau\) is an ideal of \(\g^\tau_{-1}(h) \oplus \g^\tau_0(h) \oplus \g^\tau_1(h)\). As a result, we can determine \(\g_t(h)^\tau\) by using the classification of 5-gradings and the corresponding 3-graded subalgebras for simple Lie algebras in \cite{Kan93}.

\begin{example}
  \label{ex:inv-e6(-14)}
  Consider the Lie algebra \(\g := \fe_{6(-14)}\). This Lie algebra is hermitian simple and of non-tube type.
  Up to equivalence, there exist two different involutions \(\tau \in \Aut(\g)\) with \(\tau(\Wmin) = -\Wmin\).

  The first equivalence class is characterized by \(\g^\tau \cong \sp(2,2) := \fu(2,2,\H)\). Fix a Cartan involution \(\theta\) of \(\g\) such that \(\theta\tau = \tau\theta\) and let \(h \in \g^{\tau,-\theta}\) be as in Corollary \ref{cor:invclass-5grad-inclusion}. Then \(\g_t(h) \cong \so(2,8)\), so that \(\g_t(h)^\tau\) must be isomorphic to \(\so(1,k) \oplus \so(1,8-k)\) for some \(0 \leq k \leq 8\) by \cite[Thm.\ 3.2.8]{HO97}.
  Up to isomorphy, there exists a unique 5-grading of \(\sp(2,2)\) with
  \[\sp(2,2)_{-1} \oplus \sp(2,2)_0 \oplus \sp(2,2)_1 \cong \sp(1,1) \oplus \sp(1,1) \cong \so(1,4) \oplus \so(1,4)\]
  (cf.\ \cite[Table II]{Kan93}), so that \(\g_t(h)^\tau \cong \so(1,4) \oplus \so(1,4)\).

  The second equivalence class is characterized by \(\g^\tau \cong \ff_{4(-20)}\), which is an exceptional simple Lie algebra of real rank \(1\). It has an up to isomorphy unique 5-grading with
  \[(\ff_{4(-20)})_{-1} \oplus (\ff_{4(-20)})_0 \oplus (\ff_{4(-20)})_1 \cong \so(1,8),\]
  so that \(\g_t(h)^\tau \cong \so(1,8)\).
\end{example}

By applying the same procedure as in Example \ref{ex:inv-e6(-14)} to general non-tube type hermitian simple Lie algebras, we obtain:

\begin{prop}
  \label{prop:nontube-inv-restr}
  Let \(h \in \g^\tau\) be as in {\rm Corollary \ref{cor:invclass-5grad-inclusion}}. Then the restriction of \(\tau\) to the hermitian simple Lie algebra \(\g_t(h)\) is of the following type:
  \begin{table}[H]
    \centering
    \begin{tabular}{|c|c|c|c|c|}
      \hline
      \(\g\) & \(\g_t(h)\) & \(\g^\tau\) & \(\g_t(h)^\tau\) & \(p,q\)\\
      \hline
      \(\su(p,q)\) & \(\su(q,q)\) & \(\so(p,q)\) & \(\so(q,q)\) & \(p > q > 0\) \\
      \(\su(2p,2q)\) & \(\su(2q,2q)\) & \(\sp(p,q)\) & \(\sp(q,q)\) & \(p > q > 0\) \\
      \hline
      \(\so^*(4p+2)\) & \(\so^*(4p)\) & \(\so(2p+1,\C)\) & \(\so(2p,\C)\) & \(p > 1\) \\
      \hline
      \(\fe_{6(-14)}\) & \(\so(2,8)\) & \(\sp(2,2)\) & \(\so(1,4) \oplus \so(1,4)\) & \\
      \(\fe_{6(-14)}\) & \(\so(2,8)\) & \(\ff_{4(-20)}\) & \(\so(1,8)\) & \\
      \hline
    \end{tabular}
    \caption{Restrictions of \(\tau \in \Aut(\g)\) to the tube type subalgebra \(\g_t(h)\) with \(\rk_\R \g = \rk_\R \g_t(h)\).}
    \label{table:nontube-inv-restr}
  \end{table}
\end{prop}

The following lemma will be used to reduce the construction of the subalgebras in Table \ref{table:herm-invclass} to the case where \(\g\) is of tube type:

\begin{lem}
  \label{lem:invclass-non-tube-type-5grad-red}
  Let \(h \in \fp\) be such that \(\g = \g_{-1}(h) \oplus \g_{-\frac{1}{2}}(h) \oplus \g_0(h) \oplus \g_{\frac{1}{2}}(h) \oplus \g_1(h)\) and \(\rk_\R \g_t(h) = \rk_\R \g = r\). Moreover, let \(h_0 \in \g_t(h)\) such that \([h,h_0] = 0\) and \(\spec(\ad_{\g_t(h)} h_0) = \{0, \pm \frac{1}{2}, \pm 1\}\). Then \(\spec(\ad_\g h_0) = \{0, \pm\frac{1}{2}, \pm 1\}\) and \(\g_t(h_0) \subset \g_t(h)\).
\end{lem}
\begin{proof}
  Recall from Lemma \ref{lem:tube-type-emb-lower-rank} that \(\g_t(h)\) a hermitian simple Lie algebra of tube type. Choose a maximal abelian subspace \(\fa \subset \g_t(h) \cap \fp\) containing \(h\) and \(h_0\). Then the restricted root system \(\Sigma \subset \fa^*\) of \(\g\) is of type \((BC_r)\) as in \eqref{eq:non-tube-type}. Let \(H_1,\ldots,H_r \in \fa\) be the coroots of \(2\varepsilon_1,\ldots,2\varepsilon_r\) respectively, such that \(h = \frac{1}{2}\sum_{k=1}^r H_k\) (cf.\ Remark \ref{rem:tube-type-h-element}(c)). Since \(h_0\) induces a 5-grading on \(\g_t(h)\), whose root system is of type \((C_r)\) as in \eqref{eq:tube-type}, it must be, after renumbering \(H_1,\ldots,H_r\) if necessary, of the form \(h_0 = \sum_{k=1}^s \lambda_k H_k\) with \(\lambda_k \in \{\pm \frac{1}{2}\}\) for some \(1 \leq s \leq r\). Thus, \(h_0\) also induces a 5-grading on \(\g\), and we have \(\g_t(h_0) \subset \g_t(h)\).
\end{proof}

\subsection{The proof of Theorem \ref{thm:herm-invclass}}

In this section, \(\g\) denotes a hermitian simple Lie algebra.
In order to prove Theorem \ref{thm:herm-invclass}, we first show that Table \ref{table:herm-invclass} contains all possible subalgebras \(\g(\tau,h)\) up to isomorphy:

\begin{prop}
  Let \(\tau \in \Aut(\g)\) be an involutive automorphism with \(\tau(\Wmin) = -\Wmin\) and let \(h_0 \in \g^\tau\). Then \(\g(\tau,h_0)\) is either trivial or isomorphic to one of the Lie algebras in {\rm Table \ref{table:herm-invclass}}.
\end{prop}
\begin{proof}
  Let \(\theta\) be a Cartan involution of \(\g\) commuting with \(\tau\) and let \(h_0 \in \g^\tau\). In view of Corollary \ref{cor:invclass-5grad-red}, we may assume that \(h_0\) is hyperbolic with \(\spec(\ad h_0) \subset \{0, \pm \frac{1}{2}, \pm 1\}\) and \(\g_{\pm 1}(h_0) \neq \{0\}\). As a consequence of Lemma \ref{lem:invclass-herm-conj-isom}, it suffices to consider the case where \(h_0 \in \g^{\tau, -\theta}\). By Corollary \ref{cor:invclass-5grad-inclusion}, there exists \(h \in \g^{\tau, -\theta}\) such that \(\g_{\pm 1}(h_0) \subset \g_{\pm 1}(h)\). Moreover, \(\g_t(h)\) is a tube type hermitian simple subalgebra of the same real rank as \(\g\) and we have \(h,h_0 \in \g_t(h)\) and \([h,h_0] = 0\).

Endow \(V := \g_1(h)\) with the structure of a simple euclidean Jordan algebra as in Example \ref{ex:KKT}. Then \(\g_1(h_0)\) is a Jordan subalgebra of \(V\) (cf.\ Lemma \ref{lem:caley-class-easy-case}). Thus, we have
\begin{equation}
  \label{eq:invclass-proof-5grad-red}
  \g(\tau,h_0) = (\g(\tau,h))_t(h_0) = (\g_t(h)(\tau\lvert_{\g_t(h)}, h))_t(h_0),
\end{equation}
where the first equality follows from \(\g^{-\tau}_1(h_0) \cap \Wmin(\g) \subset \g_1(h) \cap \Wmin(\g_t(h))\) and the second one from Lemma \ref{lem:invclass-nontubetype-red}.
The restriction \(-\tau\lvert_V\) is an involutive Jordan algebra automorphism by Proposition \ref{prop:invclass-jordan-inv}. We can determine whether \(-\tau\lvert_V\) is split or non-split and whether it is a Peirce reflection using Theorem \ref{thm:inv-hermlie-jordan}, respectively Proposition \ref{prop:jordan-inv-herm-rk2} if \(\g \cong \so(2,n)\) for some \(n \in \N, n \neq 2\). We can then determine \(V^{-\tau}\) up to isomorphy using \cite[Table 1.5.1]{BH98}. If \(\g\) is of non-tube type, then we use Proposition \ref{prop:nontube-inv-restr} first to determine \(\g_t(h)^\tau\) and then apply the steps mentioned before to determine \(V^{-\tau}\).

Recall from Lemma \ref{lem:jordan-herm-cone} that \(\g_1(h) \cap \Wmin(\g_t(h)) \in \{\pm \oline{\Omega_V}\}\), where \(\Omega_V\) denotes the interior of the cone of squares in \(V\). Hence, \(\Wmin(\g) \cap \g_1(h)^{-\tau}\) coincides up to a sign with the cone \(\oline{\Omega_{V^{-\tau}}}\), which generates \(V^{-\tau} = \g_1^{-\tau}(h)\). Moreover, the cone
\[\g_{-1}^{-\tau}(h) \cap -\Wmin(\g_t(h)) = -\theta(\g_1^{-\tau}(h) \cap \Wmin(\g_t(h)))\]
generates \(\g_{-1}^{-\tau}(h)\), so that \(\g_t(h)(\tau,h)\) coincides with the subalgebra generated by \(\g_{\pm 1}^{-\tau}(h)\), which is isomorphic to the Lie algebra that we obtain from \(V^{-\tau}\) via the Kantor--Koecher--Tits construction (cf.\ Remark \ref{rem:tkk-onetoone}). We can determine this Lie algebra up to isomorphy using the table in \cite[p.\ 213]{FK94}.

In order to determine \(\g(\tau,h_0)\), we first notice that \(h_0\) induces either a 3-grading or a 5-grading on \(\g_t(h)(\tau,h)\). In any case, Lemma \ref{lem:caley-class-easy-case} and \(\g_1(h_0) \subset \g_1(h)\) show that \(\g_1(h_0) \cap V^{-\tau} \cong (V^{-\tau})^{(k)}\) for some \(1 \leq k \leq r\) and a suitable Jordan frame \(F\) in \(V^{-\tau}\). Combining this with equation \eqref{eq:invclass-proof-5grad-red}, we see that \(\g(\tau,h_0)\) is the subalgebra that is obtained from the euclidean Jordan algebra \((V^{-\tau})^{(k)}\) via the Kantor--Koecher--Tits construction. We can determine \(\g(\tau,h_0)\) up to isomorphy using Table \ref{table:cayley-type-class-simplepart}. Since \(\g^\tau\) is isomorphic to one of the subalgebras in Table \ref{table:herm-invclass} for every hermitian simple Lie algebra \(\g\) and every involutive automorphism \(\tau \in \Aut(\g)\) with \(\tau(\Wmin) = -\Wmin\) (cf.\ \cite[Thm.\ 3.2.8]{HO97}), we can apply this procedure to every such \(\g\) and \(\tau\) to see that \(\g(\tau,h_0)\) is contained in Table \ref{table:herm-invclass}.
\end{proof}

It remains to show that, for a fixed \(\tau \in \Aut(\g)\), every subalgebra \(\g(\tau,h)\) in Table \ref{table:herm-invclass} does actually occur.

\begin{prop}
  For a fixed involution \(\tau \in \Aut(\g)\) with \(\tau(\Wmin) = -\Wmin\), every subalgebra \(\g(\tau,h)\) in {\rm Table \ref{table:herm-invclass}} can be realized for some \(h \in \g^\tau\).
\end{prop}
\begin{proof}
  Fix a Cartan involution \(\theta\) of \(\g\) commuting with \(\tau\).
  We first consider the case where \(\g\) is of tube type.
  By Corollary \ref{cor:invclass-5grad-inclusion}, there exists \(h \in \g^{\tau,-\theta}\) such that \(\g = \g_{-1}(h) \oplus \g_0(h) \oplus \g_1(h)\). We endow \(V := \g_1(h)\) with the structure of a simple euclidean Jordan algebra as in Example \ref{ex:KKT}.
  Then \(-\tau\) restricts to an involutive automorphism of \(V\) by Proposition \ref{prop:invclass-jordan-inv}, and by Theorem \ref{thm:invclass-3grad-jordan}, the subalgebra \(\g(\tau,h)\) is the Lie algebra that is obtained from \(V^{-\tau}\) via the Kantor--Koecher--Tits construction.
  We can determine \(\g(\tau,h)\) using Theorem \ref{thm:inv-hermlie-jordan}, Proposition \ref{prop:jordan-inv-herm-rk2}, and \cite[Table 1.5.1]{BH98}.
  With this procedure, we obtain all subalgebras \(\g(\tau,h)\) in Table \ref{table:herm-invclass} with \(\rk_\R \g = \rk_\R \g(\tau,h)\), respectively those with \(\rk_\R \g = 2\rk_\R \g(\tau,h)\) in the cases where \(2\rk_\R \g^\tau = \rk_\R \g\).
  For the case \(\g \cong \so(2,n)\), we also refer to Remark \ref{rem:jordan-inv-herm-rk2}.

  In order to obtain the subalgebras of lower rank, we proceed as follows: Let \(F = \{c_1,\ldots,c_r\}\) be a Jordan frame of \(V\) with the properties from Proposition \ref{prop:class-inv-jordan}(b) applied to \(-\tau\lvert_V\).
  Then we obtain a Jordan frame \(F_+\) of \(V^{-\tau}\) by setting \(F = F_+\) if \(-\tau\lvert_V\) is split and \(F_+ := \{c_k - \tau(c_k) : 1 \leq k \leq r\}\) otherwise (cf.\ Lemma \ref{lem:jordan-nonsplit-frame}).
  Let \(F_+ = \{d_1,\ldots,d_s\}\) for \(s = \rk V^{-\tau}\).
  We identify \(\g_0(h)\) with the structure algebra \(\str(V) = \Der(V) \oplus L(V)\) of \(V\) (cf.\ Remark \ref{rem:tkk-onetoone}).
  With a similar argument as in the proof of Lemma \ref{lem:peirce-refl-cayley-type}, we see that \(L(F_+)\) is contained in \(\g^\tau\).
  Consider now for \(1 \leq k \leq s-1\) the element \(h_k := \sum_{i=1}^k L(d_i) \in \g^\tau\). Then \((V^{-\tau})^{(k)} \cong V_1^{-\tau}(h_k) = \g_1^{-\tau}(h_k) \subset \g_1^{-\tau}(h)\), and since \(\g(\tau,h)\) is obtained from \(V^{-\tau}\) via the Kantor--Koecher--Tits construction, so is \(\g(\tau,h_k)\), and we can determine \(\g(\tau,h_k)\) using Table \ref{table:cayley-type-class-simplepart}. By applying this procedure to all such \(k\), we obtain the remaining subalgebras \(\g(\tau,h)\) in Table \ref{table:herm-invclass}.

  Suppose now that \(\g\) is of non-tube type.
  Then there exists a hyperbolic element \(h \in \g^{\tau,-\theta}\) such that \(\spec(\ad(h)) = \{0, \pm\frac{1}{2}, \pm 1\}\) and \(\g_t(h)\) is a hermitian simple subalgebra of tube type with \(\rk_\R \g = \rk_\R \g_t(h)\) (cf.\ Corollary \ref{cor:invclass-5grad-inclusion}).
  By Remark \ref{rem:tube-type-h-element}(d), we have \(\Wmin(\g_t(h)) \subset \Wmin(\g) \cap \g_t(h)\), so that \(\g(\tau,h) = \g_t(h)(\tau\lvert_{\g_t(h)},h)\).
  We can determine \(\g_t(h)^\tau\) from \(\g^\tau\) using Proposition \ref{prop:nontube-inv-restr}.
  Applying the argument from the tube type case above to \(\g_t(h)\), we obtain all subalgebras \(\g(\tau,h)\) from Table \ref{table:herm-invclass} in the cases where \(\rk_\R \g = \rk_\R \g(\tau,h)\).
  In order to obtain the subalgebras of lower rank, we construct the elements \(h_k \in \g_t(h)^{\tau, -\theta}\) for \(1 \leq k \leq s - 1\) as outlined above.
  Lemma \ref{lem:invclass-non-tube-type-5grad-red} implies that \(\g_t(h_k) \subset \g_t(h)\), so that \(\g(\tau,h_k) = \g_t(h)(\tau,h_k)\).
  Applying this procedure to every such \(k\) again yields the remaining subalgebras in Table \ref{table:herm-invclass}, which finishes the proof.
\end{proof}

\subsection*{Acknowledgement}
We thank Karl-Hermann Neeb for all the helpful discussions and his support during the research on this topic and also for reading earlier versions of the manuscript.
Moreover, we thank the referees for their detailed feedback and suggestions to improve the readability of the paper.


\begin{thebibliography}{aaaaaaaa}

\bibitem[Ar63]{Ar63} Araki, H., {\it A lattice of von Neumann algebras associated with the quantum theory of a free Bose field}, J. Math. Phys. {\bf 4} (1963), 1343--1362 

\bibitem[B57]{B57} Berger, M., {\it Les espaces sym\'etriques non-compacts}, Ann. Sci. \'Ecole Norm. Sup. (3) {\bf 74} (1957), 85--177

\bibitem[BH98]{BH98} Bertram, W., and J. Hilgert, {\it Hardy spaces and analytic continuation of Bergman spaces}, Bull. Soc. Math. France {\bf 126:3} (1998), 435--482

\bibitem[FK94]{FK94} Faraut, J., and A. Kor\'anyi, ``Analysis on symmetric cones,'' 
Oxford Math.\ Monographs, Oxford University Press, New York, 1994 

\bibitem[GT11]{GT11} Gowda, M. Seetharama, and Tao, J.,
{\it The {C}auchy interlacing theorem in simple {E}uclidean {J}ordan algebras and some consequences},
Linear and Multilinear Algebra \textbf{59:1} (2011), 65--86

\bibitem[Ha96]{Ha96} Haag, R., ``Local Quantum Physics,'' Second edition, Texts and Monographs in Physics, Springer-Verlag, Berlin, 1996

\bibitem[Hel78]{Hel78} Helgason, S., ``Differential Geometry, Lie groups, and symmetric spaces,'' Academic Press London, 1978

\bibitem[HHL89]{HHL89} Hilgert, J., K.H. Hofmann, and J.D. Lawson, 
``Lie Groups, Convex 
Cones, and Semigroups,'' The Clarendon Press, Oxford University Press, New York, 1989 

\bibitem[HN93]{HN93} Hilgert, J., and K.-H. Neeb, 
``Lie semigroups 
and their applications,'' Lecture Notes in 
Math.\ {\bf 1552}, Springer-Verlag, Berlin, 1993 

\bibitem[HN12]{HN12} Hilgert, J., and K.-H. Neeb, ``Structure and Geometry of Lie Groups,'' Springer Monographs in Mathematics, Springer, New York, 2012 

\bibitem[HNO94]{HNO94} Hilgert, J., K.-H. Neeb, and B. \O{}rsted, {\it The 
geometry of nilpotent coadjoint orbits of convex type in Hermitian Lie
algebras}, J. Lie Theory {\bf 4:2} (1994), 185--235 

\bibitem[H{\'O}97]{HO97} Hilgert, J., and G. {\'O}lafsson, ``Causal Symmetric Spaces,'' Perspectives in Mathematics {\bf 18}, Academic Press Inc., San Diego, CA, 1997

\bibitem[Hw69]{Hw69} Helwig, K.-H., {\it Halbeinfache reelle Jordan-Algebren}, Math. Z. {\bf 109} (1969), 1--28

\bibitem[Ja75]{Ja75} Jaffe, H.A., {\it Real forms of hermitian symmetric spaces}, Bull. Amer. Math. Soc. {\bf 81} (1975), 456--458

\bibitem[Kan93]{Kan93} Kaneyuki, S., {\it On the subalgebras \(\g_0\) and \(\g_{\mathrm{ev}}\) of 
semisimple graded Lie algebras}, J. Math. Soc. Japan {\bf 45:1} (1993), 1--19

\bibitem[KK85]{KK85} Kaneyuki, S., and M. Kozai, {\it Paracomplex structures and affine symmetric spaces}, Tokyo J. Math. {\bf 8:1} (1985), 81--98

\bibitem[Koe67]{Koe67} Koecher, M., {\it Imbedding of Jordan algebras into Lie algebras. I},
Amer. J. Math. {\bf 89} (1967), 787--816

\bibitem[Koe69]{Koe69} Koecher, M., ``An elementary approach to bounded symmetric domains'',
Rice University, Houston, Tex., 1969

\bibitem[KN64]{KN64} Kobayashi, S., and T. Nagano, {\it On filtered Lie algebras and geometric structures I}, \\ J. Math. Mech. {\bf 13} (1964), 875--907

\bibitem[KW65]{KW65} Kor\'{a}nyi, A., and J.A. Wolf, {\it Realization of hermitian symmetric spaces as generalized half-planes}, Ann. of Math. (2) {\bf 81} (1965), 265--288

\bibitem[Lo08]{Lo08}  Longo, R., {\it Real Hilbert subspaces, modular theory, 
SL(2, R) and CFT} in 
``Von Neumann Algebras in Sibiu'', 33-91, Theta Ser. Adv. Math. {\bf 10}, 
Theta, Bucharest, 2008

\bibitem[Ma79]{Ma79} Matsuki, T., {\it The orbits of affine symmetric spaces under the action of minimal parabolic subgroups}, J. Math. Soc. Japan {\bf 31:2} (1979), 331--357

\bibitem[Mo64]{Mo64} Moore, C. C., {\it Compactifications of symmetric spaces II}, Amer. J. Math. {\bf 86} (1964), 358--378

\bibitem[Ne00]{Ne00} Neeb, K.-H., ``Holomorphy and Convexity in Lie Theory,'' 
Expositions in Mathematics {\bf 28}, de Gruyter Verlag, Berlin, 2000 

\bibitem[Ne18]{Ne18} Neeb, K.-H., 
{\it On the geometry of standard subspaces}, 
in ``Representation theory and harmonic analysis on symmetric spaces,'' 199--223, Contemp. Math. {\bf 714}, Amer. Math. Soc., Providence, RI, 2018

\bibitem[Ne19]{Ne19} Neeb, K.-H.,
{\it Finite dimensional semigroups of unitary endomorphisms of standard subspaces},
arXiv:math.OA: 1902.02266

\bibitem[N\'O17]{NO17} Neeb, K.-H., and G. \'Olafsson,  {\it 
Antiunitary representations and modular theory}, in ``50th Seminar Sophus Lie'', 291--362, Banach Center Publ. {\bf 113}, Polish Acad. Sci. Inst. Math., Warsaw, 2017

\bibitem[\'O91]{O91} \'Olafsson, G., {\it Symmetric spaces of Hermitian type}, Differential Geom. Appl. {\bf 1:3} (1991), 195--233

\bibitem[Pa81]{Pa81} Paneitz, S. M., {\it Invariant convex cones and causality in semisimple Lie algebras and groups}, J. Functional Analysis {\bf 43:3} (1981), 313--359

\bibitem[Sa80]{Sa80} Satake, I., ``Algebraic Structures of Symmetric Domains'', Publ.\ Math.\ Soc.\ Japan {\bf 14}, Princeton Univ.\ Press, 1980

\bibitem[Wa72]{Wa72} Warner, G., ``Harmonic analysis on semi-simple Lie groups I'', Springer-Verlag, New York, 1972.

\bibitem[WK65]{WK65} Wolf, J.A., and A. Kor\'{a}nyi, {\it Generalized Cayley transformations of bounded symmetric domains}, Amer. J. Math. {\bf 87} (1965), 899--939

\end{thebibliography}
\end{document}